\documentclass[reqno]{amsart}
\usepackage{amssymb,amsmath,hyperref}
\usepackage{amsrefs}
\usepackage[foot]{amsaddr}
\usepackage{bbold,stackrel}
 
\newtheorem{thm}{Theorem}

\newtheorem{cor}[thm]{Corollary}
\newtheorem{defi}[thm]{Definition}
\newtheorem{rem}[thm]{Remark}
\newtheorem{nota}[thm]{Notation}

\newtheorem{princ}[thm]{Principle}

\newtheorem{ack}[thm]{Acknowledgement}

\newtheorem*{tempo*}{Template}
\newtheorem{theorem}[thm]{Theorem}

\newtheorem{lemma}[thm]{Lemma}
\newtheorem{definition}[thm]{Definition}
\newtheorem{corollary}[thm]{Corollary}
\newtheorem{remark}[thm]{Remark}

\newcommand\be{\begin{equation}}
\newcommand\ee{\end{equation}} 

\usepackage{amsmath,amsfonts} 

\usepackage[applemac]{inputenc}

\def\bdefi{\begin{defi}\rm}
\def\edefi{\end{defi}}
\def\bnota{\begin{nota}\rm}
\def\enota{\end{nota}}
\def\FIVE{\Pi_{1}^{1}\text{-\textup{\textsf{CA}}}_{0}}

\def\SIXk{\Pi_{k}^{1}\text{-\textsf{\textup{CA}}}_{0}}
\def\SIXK{\Pi_{k}^{1}\text{-\textsf{\textup{CA}}}_{0}^{\omega}}

\def\ATR{\textup{\textsf{ATR}}}

\def\PIT{\textup{\textsf{PIT}}}

\def\Z{\textup{\textsf{Z}}}

\def\NFP{\textup{\textsf{NFP}}}
\def\ZFC{\textup{\textsf{ZFC}}}

\def\RM{\textup{\textsf{rm}}}
\def\URY{\textup{\textsf{URY}}}
\def\TIET{\textup{\textsf{TIE}}}
\def\ZF{\textup{\textsf{ZF}}}

 \def\r{\mathbb{r}}

\def\c{\textup{\textsf{c}}}
\def\RCA{\textup{\textsf{RCA}}}
\def\({\textup{(}}
\def\){\textup{)}}

\def\RCAo{\textup{\textsf{RCA}}_{0}^{\omega}}
\def\ACAo{\textup{\textsf{ACA}}_{0}^{\omega}}

\def\WKL{\textup{\textsf{WKL}}}

\def\WWKL{\textup{\textsf{WWKL}}}
\def\bye{\end{document}}
\def\N{{\mathbb  N}}
\def\Q{{\mathbb  Q}}
\def\R{{\mathbb  R}}
\def\L{\textsf{\textup{L}}}

\def\di{\rightarrow}
\def\asa{\leftrightarrow}
\def\ACA{\textup{\textsf{ACA}}}

\def\QFAC{\textup{\textsf{QF-AC}}}

\def\HBU{\textup{\textsf{HBU}}}

\def\NIN{\textup{\textsf{NIN}}}
\def\s{\textup{\textsf{s}}}
\def\BCT{\textup{\textsf{BCT}}}
\def\WHBC{\textup{\textsf{WHBC}}}
\def\PST{\textup{\textsf{PST}}}
\def\CBT{\textup{\textsf{CBT}}}

\def\w{\textup{\textsf{w}}}

\def\closed{\textup{\textsf{closed}}}
\def\SS{\textup{\textsf{S}}}
\def\IND{\textup{\textsf{IND}}}
\def\INDD{\mathbf{IND}}

\def\CLO{\textup{\textsf{CLO}}}
\def\mTR{\textup{\textsf{-TR}}}
\def\BOOT{\textup{\textsf{BOOT}}}
\def\open{\textup{\textsf{open}}}
\def\cont{\textup{\textsf{coco}}}
\def\CLO{\textup{\textsf{CLO}}}

\def\WHBU{\textup{\textsf{WHBU}}}

\def\HBC{\textup{\textsf{HBC}}}

\def\PST{\textup{\textsf{PST}}}
\def\eps{\varepsilon}

\def\ECF{\textup{\textsf{ECF}}}

\def\SCF{\textup{\textsf{SFF}}}

\usepackage{graphicx}
\usepackage{tikz}
\usetikzlibrary{matrix, shapes.misc}

\numberwithin{equation}{section}
\numberwithin{thm}{section}

\usepackage{comment}

\begin{document}
\title{Open sets in computability theory and Reverse Mathematics}
\author{Dag Normann}
\address{Department of Mathematics, The University 
of Oslo, P.O. Box 1053, Blindern N-0316 Oslo, Norway}
\email{dnormann@math.uio.no}
\author{Sam Sanders}
\address{Department of Mathematics, TU Darmstadt, Germany}
\email{sasander@me.com}
\begin{abstract} 
To enable the study of open sets in computational approaches to mathematics, lots of extra data and structure on these sets is assumed. 
For both foundational and mathematical reasons, it is then a natural question, and the subject of this paper, what the influence of this extra data and structure is on the logical and computational properties of basic theorems pertaining to open sets.  
To answer this question, we study various basic theorems of analysis, like the Baire category, Heine, Heine-Borel, Urysohn, and Tietze theorems, all for open sets given by their (third-order) characteristic functions.
Regarding \emph{computability theory}, the objects claimed to exist by the aforementioned theorems undergo a shift from `computable' to `not computable in any type two functional', following Kleene's S1-S9.  
Regarding \emph{Reverse Mathematics}, the latter's \emph{Main Question}, namely which set existence axioms are necessary for proving a given theorem, does not have an unique or unambiguous answer for the aforementioned theorems, working in Kohlenbach's higher-order framework.  
A finer study of representations of open sets leads to the new `$\Delta$-functional' which has unique (computational) properties.    
\end{abstract}
%%
%\setcounter{page}{0}
%\tableofcontents
%\thispagestyle{empty}
%\newpage

\maketitle
\thispagestyle{empty}

%WRONG was corrected during page proofs
\section{Introduction}
Obviousness, much more than beauty, is in the eye of the beholder.  For this reason, lest we be misunderstood, we formulate a blanket caveat: all notions (computation, continuity, function, open set, comprehension, et cetera) used in this paper are to be interpreted via their well-known definitions in higher-order arithmetic listed below, \emph{unless explicitly stated otherwise}.  
\subsection{Aim and motivation}
It is a commonplace that the notion of \emph{open set} is central to topology and fundamental to large parts of mathematics in ways that few notions can boast. 
Historical analysis dates back the concept of open set to Baire's 1899 doctoral thesis, 
while Dedekind already considered this and related concepts twenty years earlier; the associated paper was published much later (\cites{moorethanudeserve, didicol}).

\smallskip

In this paper, we study open sets in computability theory (following Kleene's S1-S9; see Section \ref{HCT}) and Reverse Mathematics (RM hereafter; see Section \ref{prelim1}) in Kohlenbach's higher-order framework.  
Our motivation -in a nutshell- is that lots of extra data and structure is assumed on open sets in the various `computational' approaches to mathematics, as detailed in Remark \ref{opensetscoded}.   
For both foundational and mathematical reasons, it is then a natural question, and part of Shore's \cite{shohe}*{Problem~5.1}, what the influence of this extra data and structure is on the logical and computational properties of basic theorems pertaining to open sets. 

\smallskip

As discussed in detail in Section \ref{rmintro}, the addition of this extra data and structure has \emph{huge} consequences for (countable) open-cover compactness, but not for sequential compactness.   
For instance, the Heine-Borel theorem for countable coverings of closed sets in the unit interval is rather `mundane' when working with open sets represented via sequences of open balls (and closed sets as complements thereof).
Indeed, this second-order theorem is provable from \emph{weak K\"onig's lemma} while the finite sub-covering therein is outright computable (via an unbounded search) in terms of the other data (see \cite{simpson2}*{IV.1}, \cite{longmann}*{p.\ 459}, and \cite{brownphd}*{Lemma 3.13}).  

\smallskip

Things are quite different for $\HBC$, the Heine-Borel theorem for countable coverings of closed sets given by third-order (possibly discontinuous) characteristic functions.  
Indeed, the finite sub-covering from $\HBC$ is no longer computable in any type two functional and the other data, while no system $\SIXK$ proves $\HBC$; the former is the higher-order version of $\SIXk$ from Section \ref{HCT}.
Nonetheless, weak K\"onig's lemma \emph{plus a small fragment\footnote{The small fragment at hand is $\QFAC^{0,1}$ introduced in Section \ref{prelim1}, which is countable choice restricted to quantifier-free formulas and quantification over Baire space; this fragment is not provable in $\ZF$ and is studied in \cite{kohlenbach2}.} of countable choice} suffices to prove $\HBC$.  
As discussed in detail in Section \ref{rmintro}, we may conclude that the \emph{Main Question of RM}, namely which set existence axioms are necessary for proving the theorem, does not have an unique or unambiguous answer for $\HBC$,
% formulated with characteristic functions, 
but rather depends (heavily) on the presence of countable choice. 
We obtain similar results for other basic theorems, like the Heine, Urysohn, and Tietze theorems, and the Baire category theorem.  
We note that the study of the latter theorem requires very different proofs.  % compared to the others.  % a phenomenon we do not have an explanation for.  

\smallskip

We motivate our representation of open sets via third-order characteristic functions by the observation that e.g.\ $\R\setminus \{0\}$ has an obvious representation via a sequence of open balls, but also a discontinuous characteristic function.  
In general, open sets are coded in second-order RM by formulas involving an existential numerical quantifier (see \cite{simpson2}*{II.5.7}), and Kohlenbach has established the intimate connection between these formulas and discontinuous functions in \cite{kohlenbach2}*{\S3}.
In this light, the change from RM-codes to characteristic functions only seems like a small step, yet has an immense impact on the Heine-Borel theorem and related theorems. 

\smallskip

By the previous, a slightly different representation of open sets can have a huge effect on the associated theorems.  It is then a natural question how strong the `coding principle' is that expresses \emph{every characteristic function of an open set can be represented by a sequence of open balls}, as well as how hard it is to compute, in the sense of Kleene's S1-S9, this representation.  In both cases, the functional $\exists^{3}$ from Section \ref{HCT} suffices, but the latter yields full second-order arithmetic. 
By contrast, no system $\SIXK$ can prove this coding principle, while no type two functional can compute the countable representation from the other data.   

\smallskip

Serendipitously, a finer study of representations of open sets gives rise to the new `$\Delta$-functional'.  In a nutshell, this \emph{unqiue} functional converts between certain natural representations of open sets; $\Delta$ also has unique computational properties, discussed in detail below, in that it is natural, genuinely type three, but does not add any computational strength to discontinuous functionals like $\exists^2$ from Section~\ref{HCT} when it comes to computing functions from functions. 

\smallskip

We discuss the aforementioned results in detail in Section \ref{rmintro}.
We finish this section with a remark on representations of open sets as they are used in various `computational' approaches to mathematics. 
\begin{rem}[Open sets and representations]\label{opensetscoded}\rm
A set is \emph{open} if it contains a neighbourhood around each of its points, and an open set can be written as a \emph{countable} union of such neighbourhoods in separable spaces.    
In computational approaches to mathematics, open sets come with various `constructive' enrichments, as follows.  

\smallskip

For instance, the neighbourhood around a point of an open set is often assumed to be given together with this point (see e.g.\ \cite{bish1}*{p.\ 69}).  This is captured by our representation \eqref{R2} in Section \ref{waycool}.
Alternatively, open sets are simply represented as countable unions (called `codes' in \cite{simpson2}*{II.5.6}, `names' in \cite{wierook}*{\S1.3.4}, and `presentations' in \cite{littlefef}) of open neighbourhoods, 
i.e.\ a non-deterministic search yields the aforementioned neighbourhood of a point, as in our representation \eqref{R4}.  

\smallskip

Moreover, there are a number of `effective' results pertaining to such coded open sets, including the Urysohn lemma and Tietze theorems (see \cite{simpson2}*{II.7} and \cite{wierook}*{\S6.2}) in 
which the object claimed to exist can also be computed (in the sense of Turing) from the inputs.  
Finally, general results about open sets often require additional computational information, the most common condition being \emph{locatedness}, which is captured by our representation \eqref{R3} in Section \ref{waycool}.
A closed set $A$ is called \emph{located} if the (continuous) distance function $d(a, A):=\inf_{b\in A}d(a, b)$ exists (see \cite{bish1}*{p.\ 82}, \cite{withgusto} or \cite{twiertrots}*{p.\ 258}). Numerous sufficient conditions are known for the locatedness of (representations of) closed sets (\cite{withgusto}). 
\end{rem}
  
\subsection{Overview}\label{rmintro}
We discuss the results to be obtained in Sections \ref{rmopen} to \ref{waycool} in some detail.  We assume basic familiarity with RM and computability theory, while we refer to Section \ref{prelim} for an introduction and the necessary technical details.  

\smallskip

First of all, the equivalence between weak K\"onig's lemma and the Heine-Borel theorem for countable coverings of the unit interval is one of the early results in second-order RM, announced in \cite{fried2} and to be found in \cite{simpson2}*{IV.1.2}.
The same equivalence holds if we generalise the latter theorem to \emph{closed} subsets of the unit interval by \cite{brownphd}*{Lemma~3.13}.  Here, closed sets are the complements of open sets and an open set $U\subset \R$
is represented by some sequence of open balls $\cup_{n\in \N}B(a_{n}, r_{n})$ where $a_{n}, r_{n}$ are rationals (see \cite{simpson2}*{I.4}).  We then write the following for any $x\in \R$:
\be\label{morg}
x\in U \textup{ if and only if } {(\exists n\in \N )( |x-a_{n}|<_{\R} r_{n}  )}.  
\ee
Open and closed sets are well-studied in RM, as is clear from e.g.\ \cites{brownphd, browner, browner2}.

\smallskip

Secondly, the open $U_{0}=\R\setminus \{0\}$ can (trivially) be represented as in \eqref{morg}, but $U_{0}$ has a \emph{discontinuous} characteristic function.  
One also readily proves, under the assumption that there exists a discontinuous function on $\R$, that every open set $U$ as in \eqref{morg} has a characteristic function (see \cite{kohlenbach2}*{\S3}).   
Furthermore, Yu uses sequences of RM-codes to represent discontinuous characteristic functions in \cite{yuphd}. 
In this light, working with third-order (and thus possibly discontinuous) characteristic functions for open sets seems to stay rather close to the representation \eqref{morg} standard in second-order RM, leading to the
following definition; see Section~\ref{prelim1} for $\RCAo$. 
\bdefi[Open sets in $\RCAo$]\label{openset}
We let $Y: \R \di \R$ represent open subsets of $\R$ as follows: we write `$x \in Y$' for `$Y(x)>_{\R}0$' and call a set $Y\subseteq \R$ ‘open' if for every $x \in Y$, there is an open ball $B(x, r) \subset Y$ with $r^{0}>0$.  
A set $Y$ is called `closed' if the complement, denoted $Y^{c}=\{x\in \R: x\not \in Y \}$, is open. 
\edefi
\noindent
Note that for open $Y$ as in the previous definition, the formula `$x\in Y$' has the same complexity (modulo higher types) as \eqref{morg}, while given $(\exists^{2})$ it becomes a `proper' characteristic function, only taking values `0' and `$1$'.  Hereafter, an `open set' refers to Definition \ref{openset}, while `RM-open set' refers to \eqref{morg}.
By \cite{simpson2}*{II.7.1}, one can effectively convert between RM-open sets and (RM-codes for) continuous characteristic functions, i.e.\ 
\eqref{morg} is included in Definition \ref{openset}.  

\smallskip

Thirdly, let $\HBC$ be the higher-order theorem that for a closed set $C\subseteq [0,1]$ (as in Definition \ref{openset}) a countable covering (consisting of basic open balls) of $C$ has a finite sub-covering.  
Let $\HBC_{\RM}$ be the second-order theorem based on closed sets as in \eqref{morg}.  
We have the following, where $\Z_{2}^{\omega}\equiv \cup_{k}\SIXK$ is a higher-order version of $\Z_{2}$ and $\HBU_{\textsf{closed}}$ is $\HBC$ for \emph{uncountable} coverings (see Section \ref{HCT} for definitions).
\begin{enumerate}
 \renewcommand{\theenumi}{\alph{enumi}}
\item The system $\RCA_{0}$ proves $\HBC_{\RM}\asa \WKL $ (see e.g.\ \cite{brownphd}*{Lemma 3.13}).
\item The finite sub-covering in $\HBC_{\RM}$ is computable in terms of the closed set and the countable covering (see \cite{longmann}*{\S7.3.4}). 
\item The system $\Z_{2}^{\omega}$ cannot prove $\HBC$ while $\RCAo+\QFAC^{0,1}$ proves $\HBC\asa \WKL$ and $\RCAo$ proves $(\exists^{3})\di \HBU_{\closed}\di\HBC$ (Section \ref{rmopen}).\label{merk}
\item The finite sub-covering in $\HBC$ is not computable  in terms of the closed set, the countable covering, and any type two object (Section \ref{rmopen}).\label{kerkintm}
\end{enumerate}
Note that by the final part of item \eqref{merk}, $\HBC$ is provable \emph{without countable choice}, but much stronger comprehension axioms are needed in the absence of the latter.   
Moreover, since $\WKL+\QFAC^{0,1}$ and $\HBU_{\closed}$ are independent but have the same first-order 
strength, there is no unique set of minimal (comprehension) axioms that proves $\HBC$.  Thus, the so-called Main Question of RM from \cite{simpson2}*{p.\ 9} does not have an unique or unambiguous answer for $\HBC$.  
In Sections \ref{rmopen} and \ref{urytiet}, we show that the Heine, Urysohn, and Tietze theorems have the same properties.  Note that $\HBC$ and $\Z_{2}^{\omega}$ are \emph{third-order} in nature, 
while $\Z_{2}^{\Omega}$ is \emph{fourth-order} in nature.

\smallskip

We stress that the above `non-standard' behaviour does not apply to sequential compactness: the latter property for closed sets as in Definition \ref{openset} in the unit interval is \emph{equivalent} to $\ACA_{0}$ over Kohlenbach's $\RCAo$ by Theorem \ref{cloclo}.  
In other words, changing the coding of open sets does not seem to affect sequential compactness, but does greatly affect \emph{countable} open-cover compactness. 

\smallskip

Fourth, the previous results suggest that Definition \ref{openset}, while quite close to the original \eqref{morg}, does provide a stronger notion of open set than \eqref{morg}. 
It is then a natural question how hard it is to prove the `coding principle' $\open$ from Section~\ref{limp3} which expresses that every characteristic function as in Definition~\ref{openset} has a representation in terms of basic open balls as in \eqref{morg}.    
Moreover, one wonders how hard it is to compute (Kleene S1-S9; see Section~\ref{HCT}) this representation in terms of the other data.  We shall provide the following answers.   

\smallskip

As shown in Section~\ref{limp3}, no type two functional computes such a representation while the functional $\exists^{3}$ suffices; the latter however yields full second-order arithmetic.  
On the RM side, no system $\SIXK$ can prove $\open$, while $\Z_{2}^{\Omega}$ can.  
Moreover, $\SIXK+\open$ proves (second-order) transfinite recursion for $\Pi_{k}^{1}$-formulas, where the latter is intermediate between $\SIXk$ and $\Pi_{k+1}^{1}$-$\textsf{CA}_{0}$.
%is intermediate between $\SIXK$ and $\Pi_{k+1}^{1}$-$\textsf{CA}_{0}^{\omega}$ as this combination proves transfinite recursion for $\Pi_{k}^{1}$-formulas.  
In Section \ref{urytiet}, we also show that $\open$ is equivalent to the associated Urysohn and Tietze theorems.   

\smallskip

Fifth, another prominent theorem about open sets is the Baire category theorem, studied in all the computational approaches to mathematics from Remark \ref{opensetscoded}.
We therefore study the computational properties of the Baire category theorem in Section \ref{BCT}, based on the concept of open set as in Definition \ref{openset}.  In this way, a realiser for 
the Baire category theorem cannot be computed by any type two functional, while the inductive definition operator $\IND$ (see Section~\ref{BCT1}) can compute such a realiser, given $\exists^{2}$.   
We note that the associated proofs in this case are very different from the other proofs in this paper and the known proofs in \cite{dagsamIII, dagsamV}.

\smallskip

Sixth, there are of course representations of open sets \emph{other} than the ones provided by \eqref{morg} and Definition \ref{openset}.  
We study two such representations, called \eqref{R2} and \eqref{R3} in Section \ref{waycool}. 
Intuitively, realisers of the `$\forall \exists$'-definition of open sets are the representations as in \eqref{R2}, while in \eqref{R3} an open set is given by the function showing that the complement is located. 
In this context, we study the unique functional $\Delta$ which 
converts the former to the latter representation.  The $\Delta$-functional has surprising computational properties (see Sections \ref{prelim} and \ref{waycool} for definitions):
\begin{enumerate}
\item[(P1)] $\Delta$ is not computable in any type two functional, but computable in any Pincherle realiser, a class weaker than $\Theta$-functionals.
\item[(P2)] $\Delta$ is unique, genuinely type three, and adds no computational strength to $\exists^2$ in terms of computing functions from functions.
\end{enumerate}
In Section \ref{RDRR} we also briefly discuss the computational complexity related to the Baire category theorem and $\HBC$ under the representation \eqref{R2} from Section \ref{gintro}.

\smallskip

We finish this section with a discussion of some of our previous results from \cite{dagsamV}, which were the starting point of this paper. 
\begin{rem}[The Pincherle phenomenon]\label{qkl}\rm
Pincherle's theorem is one of the first local-to-global principles, originally proved around 1882 in \cite{tepelpinch}*{p.\ 67}, and expresses that a locally bounded function, say on Cantor space, is bounded. 
Such principles are intimately connected to compactness as claimed by Tao:
\begin{quote}
Compactness is a powerful property of spaces, and is used in many ways in many
different areas of mathematics. One is via appeal to local-to-global principles; one
establishes local control on some function or other quantity, and then uses compactness to boost the local control to global control. \cite{taokejes}*{p.\ 168}
\end{quote}
We have shown in \cite{dagsamV} that Pincherle's theorem is \emph{closely related} to (open-cover) compactness, but has \emph{fundamentally different} logical and computational properties.  
For instance, Pincherle's theorem, called $\PIT_{o}$ in \cite{dagsamV}, satisfies the following:  
\begin{itemize}
\item The system $\Z_{2}^{\omega}$ cannot prove $\PIT_{o}$, while $\RCAo+\QFAC^{0,1}$ proves $\WKL\asa \PIT_{o}$ and $\RCAo$ proves $(\exists^{3})\di \HBU\di \PIT_{o}$.  
\item Even a weak\footnote{Two kinds of realisers for Pincherle's theorem were introduced in \cite{dagsamV}: a \emph{weak Pincherle realiser} $M_{o}$ takes as input $F^{2}$ that is locally bounded on $2^{\N}$ together with $G^{2}$ such that $G(f)$ is an upper bound for $F$ in $[\overline{f}G(f)]$ for any $f\in 2^{\N}$, and outputs an upper bound $M_{o}(F, G)$ for $F$ on $2^{\N}$. A (normal) \emph{Pincherle realiser} $M_{u}$ outputs an upper bound $M_{u}(G)$ without access to $F$.} realiser for $\PIT_{o}$ cannot be computed (Kleene S1-S9) in terms of any type two functional. 
\end{itemize}
Clearly, Pincherle's theorem exhibits the same properties as $\HBC$ described in items \eqref{merk} and \eqref{kerkintm} above, and we shall therefore say that \emph{$\HBC$ exhibits the Pincherle phenomemon}, due to Pincherle's theorem $\PIT_{o}$ being the first theorem identified as exhibiting the above behaviour, namely in \cite{dagsamV}.  
In other words, part of the aim of this paper is to establish the abundance of the Pincherele phenomenon in ordinary mathematics, beyond the few examples from \cite{dagsamV}.
\end{rem}
Another way of interpreting the Pincherle phenomenon from Remark \ref{qkl} is as follows: as suggested by its title, it is claimed in \cite{kermend} that `disasters' happen in topology in absence of the Axiom of Choice, i.e.\ working in $\ZF$ or weaker fragments of $\ZFC$.  
We note that \cite{kermend} mostly deals with \emph{countable choice}.  
It is no leap of the imagination to claim that similar disasters already happen for Pincherle's theorem $\PIT_{o}$ and $\HBC$, i.e.\ in ordinary mathematics.  Indeed, Pincherle's theorem 
and $\HBC$ `should' be equivalent to weak K\"onig's lemma, or at least provable from relatively weak axioms, but it \emph{seems} this can only be guaranteed given countable choice.  

\section{Preliminaries}\label{prelim}

We introduce \emph{Reverse Mathematics} in Section \ref{prelim1}, as well as its generalisation to \emph{higher-order arithmetic}, and the associated base theory $\RCAo$.  % in Section \ref{KOH}.  
We introduce some essential axioms in Section~\ref{HCT}.  

\subsection{Reverse Mathematics}\label{prelim1}
Reverse Mathematics is a program in the foundations of mathematics initiated around 1975 by Friedman (\cites{fried,fried2}) and developed extensively by Simpson (\cite{simpson2}).  
The aim of RM is to identify the minimal axioms needed to prove theorems of ordinary, i.e.\ non-set theoretical, mathematics. 

\smallskip

We refer to \cite{stillebron} for a basic introduction to RM and to \cite{simpson2, simpson1} for an overview of RM.  We expect basic familiarity with RM, but do sketch some aspects of Kohlenbach's \emph{higher-order} RM (\cite{kohlenbach2}) essential to this paper, including the base theory $\RCAo$ (Definition \ref{kase}).  
As will become clear, the latter is officially a type theory but can accommodate (enough) set theory via Definition \ref{openset}. 

\smallskip

First of all, in contrast to `classical' RM based on \emph{second-order arithmetic} $\Z_{2}$, higher-order RM uses $\L_{\omega}$, the richer language of \emph{higher-order arithmetic}.  
Indeed, while the former is restricted to natural numbers and sets of natural numbers, higher-order arithmetic can accommodate sets of sets of natural numbers, sets of sets of sets of natural numbers, et cetera.  
To formalise this idea, we introduce the collection of \emph{all finite types} $\mathbf{T}$, defined by the two clauses:
\begin{center}
(i) $0\in \mathbf{T}$   and   (ii)  If $\sigma, \tau\in \mathbf{T}$ then $( \sigma \di \tau) \in \mathbf{T}$,
\end{center}
where $0$ is the type of natural numbers, and $\sigma\di \tau$ is the type of mappings from objects of type $\sigma$ to objects of type $\tau$.
In this way, $1\equiv 0\di 0$ is the type of functions from numbers to numbers, and  $n+1\equiv n\di 0$.  Viewing sets as given by characteristic functions, we note that $\Z_{2}$ only includes objects of type $0$ and $1$.    

\smallskip

Secondly, the language $\L_{\omega}$ includes variables $x^{\rho}, y^{\rho}, z^{\rho},\dots$ of any finite type $\rho\in \mathbf{T}$.  Types may be omitted when they can be inferred from context.  
The constants of $\L_{\omega}$ include the type $0$ objects $0, 1$ and $ <_{0}, +_{0}, \times_{0},=_{0}$  which are intended to have their usual meaning as operations on $\N$.
Equality at higher types is defined in terms of `$=_{0}$' as follows: for any objects $x^{\tau}, y^{\tau}$, we have
\be\label{aparth}
[x=_{\tau}y] \equiv (\forall z_{1}^{\tau_{1}}\dots z_{k}^{\tau_{k}})[xz_{1}\dots z_{k}=_{0}yz_{1}\dots z_{k}],
\ee
if the type $\tau$ is composed as $\tau\equiv(\tau_{1}\di \dots\di \tau_{k}\di 0)$.  
Furthermore, $\L_{\omega}$ also includes the \emph{recursor constant} $\mathbf{R}_{\sigma}$ for any $\sigma\in \mathbf{T}$, which allows for iteration on type $\sigma$-objects as in the special case \eqref{special}.  Formulas and terms are defined as usual.  
One obtains the sub-language $\L_{n+2}$ by restricting the above type formation rule to produce only type $n+1$ objects (and related types of similar complexity).        
\bdefi\label{kase} 
The base theory $\RCAo$ consists of the following axioms.
\begin{enumerate}
 \renewcommand{\theenumi}{\alph{enumi}}
\item  Basic axioms expressing that $0, 1, <_{0}, +_{0}, \times_{0}$ form an ordered semi-ring with equality $=_{0}$.
\item Basic axioms defining the well-known $\Pi$ and $\Sigma$ combinators (aka $K$ and $S$ in \cite{avi2}), which allow for the definition of \emph{$\lambda$-abstraction}. 
\item The defining axiom of the recursor constant $\mathbf{R}_{0}$: for $m^{0}$ and $f^{1}$: 
\be\label{special}
\mathbf{R}_{0}(f, m, 0):= m \textup{ and } \mathbf{R}_{0}(f, m, n+1):= f(n, \mathbf{R}_{0}(f, m, n)).
\ee
\item The \emph{axiom of extensionality}: for all $\rho, \tau\in \mathbf{T}$, we have:
\be\label{EXT}\tag{$\textsf{\textup{E}}_{\rho, \tau}$}  
(\forall  x^{\rho},y^{\rho}, \varphi^{\rho\di \tau}) \big[x=_{\rho} y \di \varphi(x)=_{\tau}\varphi(y)   \big].
\ee 
\item The induction axiom for quantifier-free formulas of $\L_{\omega}$.
\item $\QFAC^{1,0}$: the quantifier-free Axiom of Choice as in Definition \ref{QFAC}.
\end{enumerate}
\edefi
\noindent
Note that variables (of any finite type) are allowed in quantifier-free formulas of the language $\L_{\omega}$: only quantifiers are banned.
Moreover, we let $\INDD^{\omega}$ be the induction axiom for all formulas in $\L_{\omega}$.
\bdefi\label{QFAC} The axiom $\QFAC$ consists of the following for all $\sigma, \tau \in \textbf{T}$:
\be\tag{$\QFAC^{\sigma,\tau}$}
(\forall x^{\sigma})(\exists y^{\tau})A(x, y)\di (\exists Y^{\sigma\di \tau})(\forall x^{\sigma})A(x, Y(x)),
\ee
for any quantifier-free formula $A$ in the language of $\L_{\omega}$.
\edefi
As discussed in \cite{kohlenbach2}*{\S2}, $\RCAo$ and $\RCA_{0}$ prove the same sentences `up to language' as the latter is set-based and the former function-based.  Recursion as in \eqref{special} is called \emph{primitive recursion}; the class of functionals obtained from $\mathbf{R}_{\rho}$ for all $\rho \in \mathbf{T}$ is called \emph{G\"odel's system $T$} of all (higher-order) primitive recursive functionals.  

\smallskip

We use the usual notations for natural, rational, and real numbers, and the associated functions, as introduced in \cite{kohlenbach2}*{p.\ 288-289}.  
\begin{defi}[Real numbers and related notions in $\RCAo$]\label{keepintireal}\rm~
\begin{enumerate}
 \renewcommand{\theenumi}{\alph{enumi}}
\item Natural numbers correspond to type zero objects, and we use `$n^{0}$' and `$n\in \N$' interchangeably.  Rational numbers are defined as signed quotients of natural numbers, and `$q\in \Q$' and `$<_{\Q}$' have their usual meaning.    
\item Real numbers are coded by fast-converging Cauchy sequences $q_{(\cdot)}:\N\di \Q$, i.e.\  such that $(\forall n^{0}, i^{0})(|q_{n}-q_{n+i}|<_{\Q} \frac{1}{2^{n}})$.  
We use Kohlenbach's `hat function' from \cite{kohlenbach2}*{p.\ 289} to guarantee that every $q^{1}$ defines a real number.  
\item We write `$x\in \R$' to express that $x^{1}:=(q^{1}_{(\cdot)})$ represents a real as in the previous item and write $[x](k):=q_{k}$ for the $k$-th approximation of $x$.    
\item Two reals $x, y$ represented by $q_{(\cdot)}$ and $r_{(\cdot)}$ are \emph{equal}, denoted $x=_{\R}y$, if $(\forall n^{0})(|q_{n}-r_{n}|\leq {2^{-n+1}})$. Inequality `$<_{\R}$' is defined similarly.  
We sometimes omit the subscript `$\R$' if it is clear from context.           
\item Functions $F:\R\di \R$ are represented by $\Phi^{1\di 1}$ mapping equal reals to equal reals, i.e.\ extensionality as in $(\forall x , y\in \R)(x=_{\R}y\di \Phi(x)=_{\R}\Phi(y))$.\label{EXTEN}
\item The relation `$x\leq_{\tau}y$' is defined as in \eqref{aparth} but with `$\leq_{0}$' instead of `$=_{0}$'.  Binary sequences are denoted `$f^{1}, g^{1}\leq_{1}1$', but also `$f,g\in C$' or `$f, g\in 2^{\N}$'.  Elements of Baire space are given by $f^{1}, g^{1}$, but also denoted `$f, g\in \N^{\N}$'.
\item For a binary sequence $f^{1}$, the associated real in $[0,1]$ is $\r(f):=\sum_{n=0}^{\infty}\frac{f(n)}{2^{n+1}}$.\label{detrippe}
\item An object $\textbf{Y}^{0\di \rho}$ is called \emph{a sequence of type $\rho$ objects} and also denoted $\textbf{Y}=(Y_{n})_{n\in \N}$ or $\textbf{Y}=\lambda n. Y_{n}$ where $Y_{n}:=\textbf{Y}(n)$ for all $n^{0}$.
%WRONG
%\item Sets of type $\rho$ objects $X^{\rho\di 0}, Y^{\rho\di 0}, \dots$ are given by their characteristic functions $F^{\rho\di 0}_{X}\leq_{\rho\di 0}1$, i.e.\ we write `$x\in X$' for $ F_{X}(x)=_{0}1$. \label{koer} 
\end{enumerate}
\end{defi}
%WRONG
%\noindent
%The following special case of item \eqref{koer} is singled out, as it will be used frequently.
%\bdefi[$\RCAo$]\label{strijker}
%A `subset $D$ of $\N^{\N}$' is given by its characteristic function $F_{D}^{2}\leq_{2}1$, i.e.\ we write `$f\in D$' for $ F_{D}(f)=1$ for any $f\in \N^{\N}$.
%Assuming extensionality on the reals as in item \eqref{EXTEN}, we obtain characteristic functions that represent subsets of $\R$.  
%Using pairing functions, it is clear we can also represent sets of finite sequences (of reals), and relations thereon.  
%\edefi
Next, we mention the highly useful $\ECF$-interpretation. 
\begin{rem}[The $\ECF$-interpretation]\label{ECF}\rm
The (rather) technical definition of $\ECF$ may be found in \cite{troelstra1}*{p.\ 138, \S2.6}.
Intuitively, the $\ECF$-interpretation $[A]_{\ECF}$ of a formula $A\in \L_{\omega}$ is just $A$ with all variables 
of type two and higher replaced by type one variables ranging over so-called `associates' or `RM-codes' (see \cite{kohlenbach4}*{\S4}); the latter are (countable) representations of continuous functionals.  
The definition of associate may be found just below Definition \ref{FTP}.
The $\ECF$-interpretation connects $\RCAo$ and $\RCA_{0}$ (see \cite{kohlenbach2}*{Prop.\ 3.1}) in that if $\RCAo$ proves $A$, then $\RCA_{0}$ proves $[A]_{\ECF}$, again `up to language', as $\RCA_{0}$ is 
formulated using sets, and $[A]_{\ECF}$ is formulated using types, i.e.\ using type zero and one objects.  
\end{rem}
In light of the widespread use of codes in RM and the common practise of identifying codes with the objects being coded, it is no exaggeration to refer to $\ECF$ as the \emph{canonical} embedding of higher-order into second-order arithmetic. 
For completeness, we list the following notational convention for finite sequences.  
\begin{nota}[Finite sequences]\label{skim}\rm
We assume a dedicated type for `finite sequences of objects of type $\rho$', namely $\rho^{*}$, which we only use for $\rho=0, 1$.  Since the usual coding of pairs of numbers goes through in $\RCAo$, we shall not always distinguish between $0$ and $0^{*}$. 
Similarly, we do not always distinguish between `$s^{\rho}$' and `$\langle s^{\rho}\rangle$', where the former is `the object $s$ of type $\rho$', and the latter is `the sequence of type $\rho^{*}$ with only element $s^{\rho}$'.  The empty sequence for the type $\rho^{*}$ is denoted by `$\langle \rangle_{\rho}$', usually with the typing omitted.  

\smallskip

Furthermore, we denote by `$|s|=n$' the length of the finite sequence $s^{\rho^{*}}=\langle s_{0}^{\rho},s_{1}^{\rho},\dots,s_{n-1}^{\rho}\rangle$, where $|\langle\rangle|=0$, i.e.\ the empty sequence has length zero.
For sequences $s^{\rho^{*}}, t^{\rho^{*}}$, we denote by `$s*t$' the concatenation of $s$ and $t$, i.e.\ $(s*t)(i)=s(i)$ for $i<|s|$ and $(s*t)(j)=t(|s|-j)$ for $|s|\leq j< |s|+|t|$. For a sequence $s^{\rho^{*}}$, we define $\overline{s}N:=\langle s(0), s(1), \dots,  s(N-1)\rangle $ for $N^{0}<|s|$.  
For a sequence $\alpha^{0\di \rho}$, we also write $\overline{\alpha}N=\langle \alpha(0), \alpha(1),\dots, \alpha(N-1)\rangle$ for \emph{any} $N^{0}$.  By way of shorthand, 
$(\forall q^{\rho}\in Q^{\rho^{*}})A(q)$ abbreviates $(\forall i^{0}<|Q|)A(Q(i))$, which is (equivalent to) quantifier-free if $A$ is.   
\end{nota}

\subsection{Higher-order computability theory}\label{HCT}
As noted above, some of our main results will be proved using techniques from computability theory.
Thus, we first make our notion of `computability' precise as follows.  
\begin{enumerate}
\item[(I)] We adopt $\ZFC$, i.e.\ Zermelo-Fraenkel set theory with the Axiom of Choice, as the official metatheory for all results, unless explicitly stated otherwise.
\item[(II)] We adopt Kleene's notion of \emph{higher-order computation} as given by his nine clauses S1-S9 (see \cite{longmann}*{Ch.\ 5} or \cite{kleeneS1S9}) as our official notion of `computable'.
\end{enumerate}
Similar to \cites{dagsam,dagsamII, dagsamIII, dagsamV, dagsamVI}, one main aim of this paper is the study of functionals of type 3 that are \emph{natural} from the perspective of mathematical practise. 
Our functionals are \emph{genuinely} of type 3 in the sense that they are not computable from any functional of type 2.  The following definition\footnote{Stanley Wainer (unpublished) has defined the countably based functionals of finite type as an analogue of the continuous functionals, while 
John Hartley has investigated the computability theory of this type structure (see \cite{hartleycountable}). We will only need the simple cases of countably based functionals of type level at most $3$ in this paper.} is then standard in this context. 
\begin{definition}
{\em A functional $\Phi^{3}$ is \emph{countably based} if for every $F^{2}$ there is countable $X\subset \N^{\N}$ such that $\Phi(F) = \Phi(G)$ for every $G$ that agrees with $F$ on $X$.}
\end{definition}
Now, if $\Phi^{3}$ is computable in a functional of type 2, then it is countably based, but the converse does not hold. However, Hartley proves in \cite{hartleycountable} that, assuming the Axiom of Choice and the Continuum Hypothesis, if $\Phi^3$ is not countably based, then there is some $F^{2}$ such that $\exists^3$ (see below) is computable in $\Phi$ and $F$.  In other words, stating the existence of such $\Phi$ brings us 'close to' $\Z_2^\Omega$ (see below). 
In the sequel, we shall explicitly point out where we use countably based functionals.  

\smallskip

For the rest of this section, we introduce some existing functionals which will be used below.
In particular, we introduce some functionals which constitute the counterparts of second-order arithmetic $\Z_{2}$, and some of the Big Five systems, in higher-order RM.
We use the formulation from \cite{kohlenbach2, dagsamIII}.  

\smallskip
\noindent
First of all, $\ACA_{0}$ is readily derived from:
\begin{align}\label{mu}\tag{$\mu^{2}$}
(\exists \mu^{2})(\forall f^{1})\big[ (\exists n)(f(n)=0) \di [(f(\mu(f))=0)&\wedge (\forall i<\mu(f))f(i)\ne 0 ]\\
& \wedge [ (\forall n)(f(n)\ne0)\di   \mu(f)=0]    \big], \notag
\end{align}
and $\ACA_{0}^{\omega}\equiv\RCAo+(\mu^{2})$ proves the same sentences as $\ACA_{0}$ by \cite{hunterphd}*{Theorem~2.5}.   The (unique) functional $\mu^{2}$ in $(\mu^{2})$ is also called \emph{Feferman's $\mu$} (\cite{avi2}), 
and is clearly \emph{discontinuous} at $f=_{1}11\dots$; in fact, $(\mu^{2})$ is equivalent to the existence of $F:\R\di\R$ such that $F(x)=1$ if $x>_{\R}0$, and $0$ otherwise (\cite{kohlenbach2}*{\S3}), and to 
\be\label{muk}\tag{$\exists^{2}$}
(\exists \varphi^{2}\leq_{2}1)(\forall f^{1})\big[(\exists n)(f(n)=0) \asa \varphi(f)=0    \big]. 
\ee
\noindent
Secondly, $\FIVE$ is readily derived from the following sentence:
\be\tag{$\SS^{2}$}
(\exists\SS^{2}\leq_{2}1)(\forall f^{1})\big[  (\exists g^{1})(\forall n^{0})(f(\overline{g}n)=0)\asa \SS(f)=0  \big], 
\ee
and $\FIVE^{\omega}\equiv \RCAo+(\SS^{2})$ proves the same $\Pi_{3}^{1}$-sentences as $\FIVE$ by \cite{yamayamaharehare}*{Theorem 2.2}.   The (unique) functional $\SS^{2}$ in $(\SS^{2})$ is also called \emph{the Suslin functional} (\cite{kohlenbach2}).
By definition, the Suslin functional $\SS^{2}$ can decide whether a $\Sigma_{1}^{1}$-formula as in the left-hand side of $(\SS^{2})$ is true or false.   We similarly define the functional $\SS_{k}^{2}$ which decides the truth or falsity of $\Sigma_{k}^{1}$-formulas; we also define 
the system $\SIXK$ as $\RCAo+(\SS_{k}^{2})$, where  $(\SS_{k}^{2})$ expresses that $\SS_{k}^{2}$ exists.  Note that we allow formulas with \emph{function} parameters, but \textbf{not} \emph{functionals} here.
In fact, Gandy's \emph{Superjump} (\cite{supergandy}) constitutes a way of extending $\FIVE^{\omega}$ to parameters of type two.  
We identify the functionals $\exists^{2}$ and $\SS_{0}^{2}$ and the systems $\ACAo$ and $\SIXK$ for $k=0$.
We note that the operators $\nu_{n}$ from \cite{boekskeopendoen}*{p.\ 129} are essentially $\SS_{n}^{2}$ strengthened to return a witness (if existant) to the $\Sigma_{k}^{1}$-formula at hand.  %  if it exists. 

\smallskip

\noindent
Thirdly, full second-order arithmetic $\Z_{2}$ is readily derived from $\cup_{k}\SIXK$, or from:
\be\tag{$\exists^{3}$}
(\exists E^{3}\leq_{3}1)(\forall Y^{2})\big[  (\exists f^{1})(Y(f)=0)\asa E(Y)=0  \big], 
\ee
and we therefore define $\Z_{2}^{\Omega}\equiv \RCAo+(\exists^{3})$ and $\Z_{2}^\omega\equiv \cup_{k}\SIXK$, which are conservative over $\Z_{2}$ by \cite{hunterphd}*{Cor.\ 2.6}. 
Despite this close connection, $\Z_{2}^{\omega}$ and $\Z_{2}^{\Omega}$ can behave quite differently, as discussed in e.g.\ \cite{dagsamIII}*{\S2.2}.   The functional from $(\exists^{3})$ is also called `$\exists^{3}$', and we use the same convention for other functionals.  

\smallskip

Fourth, the Heine-Borel theorem states the existence of a finite sub-covering for an open covering of certain spaces. 
Now, a functional $\Psi:\R\di \R^{+}$ gives rise to the \emph{canonical covering} $\cup_{x\in I} I_{x}^{\Psi}$ for $I\equiv [0,1]$, where $I_{x}^{\Psi}$ is the open interval $(x-\Psi(x), x+\Psi(x))$.  
Hence, the uncountable covering $\cup_{x\in I} I_{x}^{\Psi}$ has a finite sub-covering by the Heine-Borel theorem; in symbols:
\be\tag{$\HBU$}
(\forall \Psi:\R\di \R^{+})(\exists  y_{1}, \dots, y_{k}\in I)(\forall x\in I)( x\in \cup_{i\leq k}I_{y_{i}}^{\Psi}).
\ee
Note that $\HBU$ is almost verbatim \emph{Cousin's lemma} (\cite{cousin1}*{p.\ 22}), i.e.\ the Heine-Borel theorem restricted to canonical coverings.  
This restriction does not make a big difference, as shown in \cite{sahotop}.
By \cite{dagsamIII, dagsamV}, $\Z_{2}^{\Omega}$ proves $\HBU$ but $\Z_{2}^{\omega}+\QFAC^{0,1}$ cannot, 
and basic properties of the \emph{gauge integral} (\cite{zwette, mullingitover}) are equivalent to $\HBU$.  

\smallskip

Fifth, since Cantor space (denoted $C$ or $2^{\N}$) is homeomorphic to a closed subset of $[0,1]$, the former inherits the same property.  
In particular, for any $G^{2}$, the corresponding `canonical covering' of $2^{\N}$ is $\cup_{f\in 2^{\N}}[\overline{f}G(f)]$ where $[\sigma^{0^{*}}]$ is the set of all binary extensions of $\sigma$.  By compactness, there are $ f_0 , \ldots , f_n \in 2^{\N}$ such that the set of $\cup_{i\leq n}[\bar f_{i} G(f_i)]$ still covers $2^{\N}$.  By \cite{dagsamIII}*{Theorem 3.3}, $\HBU$ is equivalent to the same compactness property for $C$, as follows:
\be\tag{$\HBU_{\c}$}
(\forall G^{2})(\exists  f_{1}, \dots, f_{k} \in C ){(\forall f \in C)}(f\in\cup_{i\leq k} [\overline{f_{i}}G(f_{i})]).
\ee
We now introduce the specification $\SCF(\Theta)$ for a (non-unique) functional $\Theta$ which computes a finite sequence as in $\HBU_{\c}$.  
We refer to such a functional $\Theta$ as a \emph{realiser} for the compactness of Cantor space, and simplify its type to `$3$'.  
\be\tag{$\SCF(\Theta)$}
(\forall G^{2})(\forall f^{1}\leq_{1}1)(\exists g\in \Theta(G))(f\in [\overline{g}G(g)]).
\ee
%\edefi
Clearly, there is no unique such $\Theta$ (just add more binary sequences to $\Theta(G)$) and any functional satisfying the previous specification 
is referred to as a `$\Theta$-functional' or a `special fan functional'.
As to their provenance, $\Theta$-functionals were introduced as part of the study of the \emph{Gandy-Hyland functional} in \cite{samGH}*{\S2} via a slightly different definition.  
These definitions are identical up to a term of G\"odel's $T$ of low complexity by \cite{dagsamII}*{Theorem 2.6}.  

\smallskip

Sixth, a number of higher-order axioms are studied in \cite{samph} including the following comprehension axiom (see also Remark \ref{hist}):
\be\tag{$\BOOT$}
(\forall Y^{2})(\exists X\subset \N)\big(\forall n\in \N)(n\in X\asa (\exists f^{1})(Y(f, n)=0)\big).
\ee
We only mention that this axiom is equivalent to e.g.\ the monotone convergence theorem for nets indexed by Baire space (see \cite{samph}*{\S3}).  
As it turns out, the coding principle $\open$ (see Section \ref{limp3}) is closely related to $\BOOT$ and fragments, as shown in Section \ref{rmlimp3}. 
Finally, we mention some historical remarks related to $\BOOT$.
\begin{rem}[Historical notes]\label{hist}\rm
First of all, $\BOOT$ is called the `bootstrap' principle as it is weak in isolation (equivalent to $\ACA_{0}$ under $\ECF$, in fact), but becomes much stronger 
when combined with comprehension axioms: $\SIXK+\BOOT$ readily proves $\Pi_{k+1}^{1}\textup{-}\textsf{CA}_{0}$.

\smallskip

Secondly, $\BOOT$ is definable in Hilbert-Bernays' system $H$ from the \emph{Grundlagen der Mathematik} (see \cite{hillebilly2}*{Supplement IV}).  In particular, one uses the functional $\nu$ from \cite{hillebilly2}*{p.\ 479} to define the set $X$ from $\BOOT$.  
In this way, $\BOOT$ and subsystems of second-order arithmetic can be said to `go back' to the \emph{Grundlagen} in equal measure, although such claims may be controversial.  

\smallskip

Thirdly, after the completion of \cite{samph}, it was observed by the second author that Feferman's `projection' axiom \textsf{(Proj1)} from \cite{littlefef} is similar to $\BOOT$.  The former is however formulated using sets (and set parameters), which makes it more `explosive' than $\BOOT$ in that full $\Z_{2}$ follows when combined with $(\mu^{2})$, as noted in \cite{littlefef}*{I-12}.
\end{rem}

\section{The Heine-Borel theorem}\label{rmopen}
\subsection{Introduction and preliminaries}
In this section, we establish the results sketched in Section \ref{rmintro} for the (countable) Heine-Borel theorem and related (countable) Vitali covering theorem.
Hereafter, we use the notion of open (and closed) set as outlined in Definition~\ref{openset}, unless explicitly stated otherwise.

\smallskip

In particular, we show that sequential compactness behaves `as normal' in this section, but that Heine-Borel compactness \emph{for countable coverings by intervals} now behaves \emph{quite} out of the ordinary, namely that it suffers from the Pincherle phenomenon (see Section \ref{kier}).  In Section~\ref{cerucc},  we establish the computability-theoretic results in item~\eqref{merk} from Section \ref{rmintro} and some related results.  All the aforementioned deals with countable coverings consisting of open \emph{intervals}, and we study coverings consisting of arbitrary open sets in Section \ref{storng}.
The latter results are complimentary to the rest and endow our equivalences with a certain robustness. 
We remind the reader that `$C$' and `$2^{\N}$' are used interchangeably for Cantor space, but this use is always carefully separated from the generic term `closed set $C\subseteq \R$'.

\smallskip

First of all, Theorem \ref{cloclo} is a sanity check for Definition \ref{openset}: our closed sets have the same properties as RM-codes for closed sets, as follows. 
\begin{enumerate}
 \renewcommand{\theenumi}{\alph{enumi}}
\item RM-closed sets are sequentially closed, i.e.\ if a sequence in an RM-closed set converges to some limit, the latter is also in the set (trivial in $\RCA_{0}$).  
\item RM-closed sets in $[0,1]$ are sequentially compact in $\ACA_{0}$ (\cite{brownphd}*{Lemma 3.14}).
\item Given a sequence in an RM-closed set in $[0,1]$, $\exists^{2}$ computes the limit (\cite{yamayamaharehare}).
\end{enumerate}
The following theorem shows that our closed sets mirror these three items perfectly.  
\begin{thm}\label{cloclo}
The system $\RCAo$ proves that a closed set is sequentially closed.  The system $\RCAo$ proves the equivalence between $\ACA_{0}$ and the statement \emph{a closed set in $[0,1]$ is sequentially compact}.  
The functional $\exists^{2}$ computes an accumulation point of a sequence in a closed set in $[0,1]$. % and has a supremum.  
\end{thm}
\begin{proof}
For the first part, if a sequence $\lambda n.x_{n}$ in a closed set $C\subset \R$ converges to $y\in \R$, but $y\not\in C$, then there is $N^{0}$ such that $B(y, \frac{1}{2^{N}})\subset C^{c}$, as the complement of $C$ is open. 
However, $x_{n}$ is eventually in $B(y, \frac{1}{2^{N}})$ by definition, a contradiction. 

\smallskip

For the second part, the reversal follows from considering the unit interval and \cite{simpson2}*{III.2.2}.
For the forward direction, the usual proof of the Bolzano-Weierstrass theorem as in \cite{simpson2}*{III.2.1} goes through, modulo using $(\exists^{2})$ to decide elementhood of closed sets. 
In case $\neg(\exists^{2})$, all $\R\di \R$-functions are continuous by \cite{kohlenbach2}*{\S3} and closed sets in $[0,1]$ reduce to the usual RM-definition by \cite{simpson2}*{II.5.7} and \cite{kohlenbach4}*{Prop.~4.10}.  
The second-order proof from \cite{brownphd}*{Lemma 3.14} now finishes this case.  The law of excluded middle $(\exists^{2})\vee \neg(\exists^{2})$ finishes the proof.  
\end{proof}
The previous theorem can be generalised to other theorems pertaining to sequential compactness (see e.g.\ \cite{simpson2}*{III}), like the monotone convergence theorem for sequences in closed subsets of the unit interval. 

\subsection{Heine-Borel theorem for basic coverings}\label{kier}
Theorem \ref{cloclo} shows that our notion of closed sets has the usual properties when it comes to sequential compactness.  In this section, we show that the situation is \emph{markedly} different for Heine-Borel compactness: Theorems \ref{dich} and \ref{goodd} namely show that $\HBC$, defined as follows, suffers from the Pincherle phenomenon.  
\bdefi[$\HBC$]
Let $C\subseteq [0,1]$ be closed and let $(a_{n})_{n\in \N}$, $(b_{n})_{n\in \N}$ be sequences of reals with $C\subseteq \cup_{n\in \N}(a_{n}, b_{n})$.  Then $C\subseteq \cup_{n\leq n_{0}}(a_{n}, b_{n})$ for some $n_{0}\in \N$.
\edefi
\noindent
We let $\HBC_{\RM}$ be $\HBC$ with $C$ represented by RM-codes.  
By Theorem \ref{dich}, $\HBC$ is provable without countable choice and has weak first-order strength.  
Indeed, $\HBU_{\closed}$ is $\HBU$ generalised to closed sets $C\subseteq [0,1]$, and both have the first-order strength of $\WKL$; this follows from applying $\ECF$ and noting \cite{brownphd}*{Lemma~3.13}.
Furthermore, $\Z_{2}^{\Omega}$ proves $\HBU_{\closed}$ in the same way as in \cite{dagsamV}*{Theorem 4.2}.
\begin{thm}\label{dich}
Both $\RCAo+\WKL+\QFAC^{0,1}$ and $\RCAo+\HBU_{\closed}$ prove $\HBC$.
\end{thm}
\begin{proof}
For the first part, in case $\neg(\exists^{2})$, all functions on $\R$ are continuous by \cite{kohlenbach2}*{\S3}.  Following the results in \cite{kohlenbach4}*{\S4}, continuous functions on $[0,1]$ have an RM-code, i.e.\ 
our definition of open set reduces to an $\L_{2}$-formula in $\Sigma_{1}^{0}$, which (equivalently) defines a code for an open set by \cite{simpson2}*{II.5.7}.  In this way, $\HBC$ is merely $\HBC_{\RM}$, which follows
from $\WKL$ by \cite{brownphd}*{Lemma 3.13}.    
In case $(\exists^{2})$, let $C\subseteq [0,1]$ be a closed set and let $(a_{n})_{n \in \N}, (b_{n})_{n\in\N}$ be as in $\HBC$.  
If there is no finite sub-covering, then $(\forall m^{0})(\exists x\in C)\big[x\not \in  \cup_{n\leq m}(a_{n}, b_{n})\big]$.  Apply $\QFAC^{0,1}$ and $(\exists^{2})$ to obtain a sequence $\lambda n.x_{n}$ of reals in $C$ with this property.
Since $(\exists^{2})\di \ACA_{0}$, any sequence in $[0,1]$ has a convergent sub-sequence $\lambda n.y_{n}$ by \cite{simpson2}*{III.2}.  If $y_{n}$ converges to $y\not \in C$, then there is $N^{0}$ such that $B(y, \frac{1}{2^{N}})\subset C^{c}$, as the complement of $C$ is open by definition. 
However, $y_{n}$ is eventually in $B(y, \frac{1}{2^{N}})$ by definition, a contradiction.   
Hence, $\lim_{n\di \infty}y_{n}=y\in C$ but if $y\in (a_{k}, b_{k})$, then $y_{n}$ is eventually in this interval, which contradicts the definition of $x_{n}$ (and $y_{n}$). 
The law of excluded middle now finishes the proof.

\smallskip

For the second part, let $C\subseteq [0,1]$ be a closed set and let $a_{n}, b_{n}$ as in $\HBC$.  Similar to the first case, we may assume $(\exists^{2})$. 
Apply $\QFAC^{1,0}$ and $(\exists^{2})$ to $(\forall x\in C)(\exists n^{0})(x\in (a_{n}, b_{n}))$ to obtain $\Psi^{2}$ yielding $n^{0}$ from $x\in C$.  
Then $\cup_{x\in C}(a_{\Psi(x)}, b_{\Psi(x)})$ is a covering of $C$ that is readily converted to a canonical covering.  We now obtain $\HBC$ from applying $\HBU_{\closed}$.
\end{proof}
\begin{cor}\label{simpleright}
The system $\RCAo+\QFAC^{0,1}$ proves $\WKL\asa \HBC$.
\end{cor}
\begin{proof}
The reverse direction follows from \cite{simpson2}*{II.7.1} as the latter provides (RM-codes for) continuous characteristic functions for RM-open sets (and vice versa). 
These codes become type two functionals by $\QFAC^{1,0}$, included in $\RCAo$.
Alternatively, use $C=[0,1]$ and the results in \cite{simpson2}*{IV.1}.
\end{proof}
In contrast to sequential compactness and Theorem \ref{cloclo}, a proof of countable Heine-Borel compactness as in $\HBC$ seems to either require countable choice \emph{or} lots of comprehension, i.e.\ we observe the Pincherle phenomenon.
\begin{thm}\label{goodd}
The system $\Z_{2}^{\omega}$ cannot prove $\HBC$.
\end{thm}
\begin{proof}
We will modify the method used to prove \cite{dagsamIII}*{Theorem 3.4} and \cite{dagsamV}*{Theorem~4.8}.
As in those proofs, let $A = \cup_{k \in \N}A_k$ be a countable subset of $\N^{\N}$ defined by the following three conditions.
\begin{enumerate}
\item[(i)] We have that $\Pi^1_n$-formulas are absolute for $(A,\N^\N)$ for all $n$.
\item[(ii)] For all $k$, we have $A_k \subseteq A_{k+1}$ and there is $f_k \in A_{k+1}$ enumerating $A_k$.
\item[(iii)] There is a sequence $g_1, g_2, … $ in $A$ such that for all $k$ there is a number $n_k$ such that $A_k$ is the computational closure of $g_0 , \ldots , g_{n_k}$ relative to $\SS^2_k$.
\end{enumerate}
In a nutshell, we shall define a model $\mathcal{M}$ and a set $X$ therein, such that the former thinks the latter is closed. 
We also identify a countable covering $\cup_{n\in \N}(a_{n}, b_{n})$ of $X$ that has no finite sub-covering inside $\mathcal{M}$, i.e.\ $\HBC$ fails in $\mathcal{M}$.  

\smallskip

Let ${\bf C } \subset [0,1]$ be the classical Cantor set (available in $\RCA_{0}$ by \cite{simpson2}*{I.8.6}) and let $\{(a_i,b_i)\}_{i \in \N}$ be a computable enumeration of the set of  open intervals forming the complement of ${\bf C}$.  Fix an element $x \in {\bf C}$ that has no code in $A$. For each $k$, define $x_k$ as an element with a code in $A_{k+1} \setminus A_k$ such that $x_k \not \in {\bf C}$ and such that the distance from $x$ to $x_k$ is bounded by $2^{-k}$. This construction is possible as the set of elements in $[0,1]$ with codes in $A_{k+1} \setminus A_k$ is dense. 

\smallskip

Define $X = \{x_k : k \in \N\}$ and consider the type-structure ${\mathcal M}$ obtained by closing $A$ and $X$ under computability relative to all $\SS^2_k$. In computations relative to $\SS^2_k$ and $X$, the two  sets $X$ and $\{x_l : l < k\}$ will give the same oracle answers, so $A_k$ will be closed under computability relative to $\SS^2_k$ and $X$. 

\smallskip

We now show that $X$ is a closed set from the point of view of $\mathcal{M}$.
To this end, consider  $y \in\big( [0,1]\setminus X\big)$ with a code  in $A$. Fix $\epsilon: = |x-y] > 0$ and let $k$ be so large that $2^{k-1} < \epsilon$. 
Then at most $x_i$ for $i < k$ will be in $\epsilon$-distance from $y$ and thus $y$ has a positive distance to $X$. In this way, $\mathcal{M}$ satisfies that $X$ is closed.

\smallskip

Finally, the computable family $\{(a_i,b_i)\}_{i \in \N}$ is in ${\mathcal M}$ and covers $X$.  Now, in the real world, $X$ has a cluster point in ${\bf C}$ which is \emph{different} from all $a_i$ and $b_i$.
Hence, there cannot be a finite sub-covering of $X$, and there is also no finite sub-covering of $X$ in ${\mathcal M}$.  As a result, $\HBC$ fails in $\mathcal{M}$ and we are done. 
\end{proof}
\begin{cor}\label{hantonio}
A proof of $\HBC$ goes through in $\WKL+\QFAC^{0,1}$ or $\Z_{2}^{\Omega}$, while $\Z_{2}^{\omega}$ does not suffice. 
\end{cor}
The negative results from the theorem also yield the following nice disjunction.  
\begin{cor}\label{scor}
The system $\RCAo$ proves $\WKL_{0}\asa [\ACA_{0}\vee \HBC]$.
\end{cor}
\begin{proof}
The reverse implications are immediate by the previous.  
The forward implication follows from $(\exists^{2})\vee \neg(\exists^{2})$ and the proof of Theorem \ref{dich}. 
\end{proof}
The previous results are not an isolated incident, as witnessed by the following.  
Note that items \eqref{bounk} to \eqref{bounk4} are studied in \cite{simpson2}*{VI.2} for RM-codes of closed sets, 
while e.g.\ items \eqref{bounk3}, \eqref{bounk4}, and \eqref{bounk6} for RM-codes are studied in \cite{browner}*{\S4}.
\begin{thm}\label{noway}
The following theorems imply $\HBC$ over $\RCAo$:
\begin{enumerate}
 \renewcommand{\theenumi}{\alph{enumi}}
\item Pincherle’s theorem for $[0,1]$\textup{:} a locally bounded function on $[0,1]$ is bounded.\label{mikkel}  
\item If $F^{2}$ is continuous on a closed set $D\subset 2^{\N}$, it is bounded on $D$.\label{bounk}
\item If $F^{2}$ is continuous on a closed set $D \subset 2^{\N}$, it is uniformly cont.\ on $D$.\label{bounk2}
\item If $F$ is continuous on a closed set $D\subset [0,1]$, it is bounded on $D$.\label{bounk0}
\item If $F$ is continuous on a closed set $D \subset [0,1]$, it is uniformly cont.\ on $D$.\label{bounk3}
\item If $F$ is continuous on a closed set $D\subset [0,1]$, it attains a maximum on $D$.\label{bounk4}
\item If $F$ is continuous on a closed set $D\subset [0,1]$, then for every $\eps>0$ there is a polynomial $p(x)$ such that $|p(x)-F(x)|<\eps$ for all $x\in D$. \label{bounk6}
\end{enumerate}
\end{thm}
\begin{proof}
By the results in \cite{dagsamV}*{\S4} and \cite{simpson2}*{IV}, all items from the theorem imply $\WKL$. 
Moreover, in case $\neg(\exists^{2})$, closed sets reduce to RM-codes for closed sets, i.e.\ $\HBC$ is just $\HBC_{\RM}$, which follows from $\WKL$ by \cite{brownphd}*{Lemma 3.13}.  
Similarly, items \eqref{bounk} and \eqref{bounk2} reduce to their second-order counterparts, equivalent to $\WKL_{0}$ by \cite{simpson2}*{IV.2}. 
Hence, we may assume $(\exists^{2})$ in the following. 
For item \eqref{mikkel}, let $a_{n}, b_{n}$ be as in the antecedent of $\HBC$, i.e.\ $(\forall x\in C)(\exists n^{0})(x\in (a_{n}, b_{n}))$.  
Applying $\QFAC^{1,0}$ and $(\exists^{2})$ (to decide whether $x\in C$ or not), one obtains $\Phi^{2}$ such that $\Phi(x)$ is the least such $n$ if $x\in C$.
By definition, $\Phi$ satisfies $\Phi(y)\leq \Phi(x)$ for any $x\in C$, $y\in C\cap B(x, r)$, and small enough $r>_{\R}0$.  Hence, the function $f:\R \di \R$ defined 
as $\Phi(x)$ if $x\in C$ and $1$ otherwise, is locally bounded on $[0,1]$.  By Pincherle's theorem, $f$ is bounded on $[0,1]$, implying that $\Phi$ is bounded on $C$ 
and immediately yielding a finite sub-covering for $\cup_{n\in \N}(a_{n}, b_{n})$, as required for $\HBC$.

\smallskip

For item \eqref{bounk}, since Cantor space is homeomorphic to a closed subset of $[0,1]$, $\HBC$ is equivalent to $\HBC$ for Cantor space.  
Let $D\subseteq 2^{\N}$ be closed and let $\lambda n.\sigma_{n}^{0\di 0^{*}}$ be a sequence of finite binary sequences covering $D$, i.e.\ $(\forall f\in D)(\exists n^{0})(f\in [\sigma_{n}])$.
Applying $\QFAC^{1,0}$ and $(\exists^{2})$ (to decide whether $x\in D$ or not), one obtains $\Phi^{2}$ such that $\Phi(f)$ is the least such $n$ if $x\in D$.
Define $G(f):=|\sigma_{\Phi(f)}|$ and note:
\be\label{reg}
(\forall f, g\in 2^{\N})(\overline{f}G(f)=\overline{g}G(f)\di \Phi(f)=\Phi(g)),
\ee
i.e.\ $\Phi$ is continuous with modulus of continuity $G$.  Item \eqref{bounk} proves that $\Phi$ is bounded on $D$, yielding a finite sub-covering of $\cup_{n\in \N}[\sigma_{n}]$.
Item \eqref{bounk2} now also readily implies $\HBC$.  For items \eqref{bounk0} and \eqref{bounk3}, since closed sets in Cantor space are also closed sets in $[0,1]$, these items follows from the items \eqref{bounk} and \eqref{bounk2}.  
Finally, note that either of items \eqref{bounk4} or \eqref{bounk6} imply item \eqref{bounk0}. 
\end{proof}
Note that by the theorem, the \emph{Tietze \(extension\) theorem} for closed sets as in Definition \ref{openset}, is not provable in $\Z_{2}^{\omega}$.  
We establish sharper results in Section \ref{urytiet}.  Moreover, $\ACA_{0}\asa [(\exists^{2})\vee (\textup{item \eqref{bounk4}})]$ follows in the same way as for Corollary \ref{scor}.

\smallskip

Finally, regarding \eqref{reg}, Kohlenbach shows in \cite{kohlenbach4}*{\S4} that over $\RCAo$ the existence of a modulus of continuity is equivalent to the existence of an RM-code, i.e.\ the exact 
formulation of continuity does not matter in the previous theorem.

\subsection{Computability theory and related results}\label{cerucc}
We establish the results in item \eqref{kerkintm} from Section \ref{rmintro} pertaining to $\HBC$, and discuss related results, including the weaker \emph{Vitali covering lemma}. 
We define a realiser for $\HBC$ as follows.
\bdefi A functional $\beta^{3}$ is called a \emph{realiser} for $\HBC$ if for closed $C\subset [0,1]$ and sequences of rationals $(a_{n})_{n\in \N}, (b_{n})_{n\in \N}$ 
such that $C\subseteq \cup_{n\in \N}(a_{n}, b_{n})$, we also have $C\subseteq \cup_{n\leq \beta(C, a_{n}, b_{n})}(a_{n}, b_{n})$.
\edefi
The following theorem is not that surprising in light of some of our previous results.  We shall establish a more impressive result in Theorem \ref{cikometric}.
\begin{thm}\label{reverand}
No type two functional can compute a realiser $\beta^{3}$ for $\HBC$.
\end{thm}
\begin{proof}
The functional $\beta$ is easily seen to be \emph{not} countably based, namely as follows. Now put $ D =  [\frac{1}{3},\frac{2}{3}]$, $a_n = \frac{1}{n}$ and $b_n = 1$. If $X$ is a countable base for the value
$\beta(D,a_n,b_n) = k$, we may choose $x \not \in X$ in the interval $(0,\frac{1}{k})$ and obtain a contradiction by considering the set $D' = D \cup \{x\}$.
\end{proof}
Finally, it is a natural question whether Theorem \ref{goodd} holds for weaker principles, e.g.\ from the RM of $\WWKL$; the latter is the restriction of $\WKL$ to trees of positive measure and equivalent to Vitali's covering theorem for countable coverings (see \cite{simpson2}*{X.1}).
Towards a positive answer, consider the following covering theorem $\WHBC$.  Note that the latter is just \cite{sayo}*{Lemma 8, item 2} formulated using Definition \ref{openset}.   
\bdefi[$\WHBC$]
Let $\eps>0$, let $C\subseteq [0,1]$ be a closed set, and let $(a_{n})_{n\in \N}$, $ (b_{n})_{n\in \N}$ be sequences of reals such that $C\subseteq \cup_{n\in \N}(a_{n}, b_{n})$.  Then there are $n_{0}\in \N$ and reals $c_{0}, d_{0},\dots, c_{m_{0}}, d_{m_{0}} $ such that $C\subseteq \bigcup_{n\leq n_{0}}(a_{n}, b_{n})\cup \bigcup_{m\leq m_{0}}(c_{m}, d_{m})$ and $\sum_{m\leq m_{0}}|c_{m}-d_{m}|<\eps$.
\edefi
On one hand, $\ECF$ clearly converts $\WHBC$ into a theorem at the level of $\WWKL$.
On the other hand, the equivalence $\WWKL\asa \WHBC$ can be proved as in Corollary~\ref{simpleright}, \emph{additionally assuming} the coding principle that every continuous functional on Baire space has a code (see \cite{kohlenbach4}*{\S4} for this coding principle). 
We have the following improvement of Theorem \ref{goodd}.
\begin{thm}
The system $\Z_{2}^{\omega}$ cannot prove $\WHBC$
\end{thm}
\begin{proof} 
The following proof is a (considerable) variation on the proof of Theorem~\ref{goodd}.  

\smallskip

First of all, we provide a sketch, as follows: let $A, A_k$ be as in the proof of Theorem \ref{goodd}. 
Analogously, we shall define a certain sequence $\{x_k\}_{k \in \N}$ of codes in  $A$ such that $x_k$ has a code in $A_{k+1} \setminus A_k$, and we let $X$ be the set of all $x_k$. 
We ensure that a function is in $A$ if it computable from $X$, elements in $A$, and some $\SS^2_k$.  With this in place, we construct the model $\mathcal{M}$ as in the proof of Theorem \ref{goodd}. 

\smallskip
\noindent
Secondly, our construction of $X$ shall satisfy the following requirements.
\begin{enumerate}
\item The set $X$ is covered by a computable sequence of $(a_i,b_i)$ of disjoint open intervals with a total measure of $\leq \frac{1}{2}$.\label{kem1}
\item The set of cluster-points of $X$ has measure $\geq \frac{1}{4}$.\label{kem2}
\item There is no cluster point of $X$ in $A$ or in any of the intervals $(a_i,b_i)$.\label{kem3}
\end{enumerate}
We now establish that items \eqref{kem1}-\eqref{kem3} guarantee that $\WHBC$ fails in the model ${\mathcal M}$ for $\epsilon < \frac{1}{4}$.
Indeed, assume that $\{(a_1 , b_1), \ldots , (a_n,b_n),(c_1 , d_1), \ldots , (c_m,d_m)\}$ covers $X$ where all $c_j$ and $d_j$ are in $\mathcal M$ and $\eps<\sum_{1 \leq j \leq m}(d_j -c_j)$. By item \eqref{kem3}, no cluster point of $X$ is contained in an $(a_i,b_i)$ or equals a $c_j$ or $d_j$, for $i \leq n$ or $j \leq m$. It follows that every cluster point of $X$ must be in some $(c_j,d_j)$, since it otherwise will have a positive distance to a set containing $X$ as a subset. However, this is in conflict with the choice of $\epsilon < \frac{1}{4}$ and that the measure of the set of cluster points is larger than $\frac{1}{4}$.

\smallskip

Finally, we construct a set $X$ satisfying items \eqref{kem1}-\eqref{kem3}.  Let $K \subseteq [0,1]$ be homeomorphic to the Cantor set, but of measure $\geq \frac{1}{2}$. It is straightforward to construct a set like this such that the complement is the pairwise disjoint union of a computable sequence $\{(a_i,b_i)\}_{i \in \N}$ of open intervals with rational endpoints. Let $\{y_k\}_{k \in \N}$ be an enumeration of $A \cap K$. We first select closed  intervals $I_i$ with $y_i$ in the interior and with length $< 2^{-(i + 3)}$, ensuring that the union of these intervals has measure $< \frac{1}{4}$. 
When selecting $x_k$, we make sure that if $x_k \in I_i$, then $k < i$ . This ensures that $y_i$ is not a cluster point of the final set $X$. 

\smallskip

Next, define $ Z$ as the set of $z \in K$ such that  $z \not \in \bigcup_{i \in \N}I_i$; $Z$ clearly has measure at least $\frac{1}{4}$. We shall now define $X$ in such a way that all elements of $Z$ are cluster points of $X$ and that all cluster points of $X$ are in $K$. To this end, let $\{z_j : j \in \N\}$ be a countable and dense subset of $Z$. It suffices to show that all $z_j$ are cluster-points of $X$ and that all cluster-points of $X$ are in $K$. 

\smallskip

As noted above, $X$ is a sequence $\{x_{k}\}_{k\in \N}$, and we now define the latter.  
To define $x_{k}$ for $k = \langle j,n \rangle$, we note that
any open interval containing $z_j$ will intersect $[0,1] \setminus K$, since $K$ contains no open intervals.  
Thus, we can find $y$ such that $y \not \in K$, $y \in A_{k+1}\setminus A_k$, the distance from $y$ and $z_j$ is less than $2^{-k}$, and $y \not \in I_i$ for $i \leq k$. 
Define $x_{k}$ as this $y$ and note that the resulting set $X$ satisfies item \eqref{kem1}-\eqref{kem3}.
\end{proof} 
One can obtain similar results for other theorems related to the RM of $\WWKL$.  
For instance, the $\L_{2}$-statement \emph{every closed set of positive measure contains a perfect sub-set} is intermediate between $\WKL_{0}$ and $\WWKL_{0}$ (\cite{wangetal}).
The generalisation of this statement involving Definition \ref{openset} should imply $\WHBC$.    
Moreover, the conclusion of the latter essentially expresses that the measure of $[0,1]\setminus \cup_{n\in \N}(a_{n}, b_{n})$ is $0$;  one could strengthen this to Borel's notion `strongly measure zero' (\cites{borelstrong}).
One could also study `probabilistic choice', which for $2^{\N}$ and $[0,1]$ is equivalent to $\WWKL$ in the Weihrauch lattice (see \cite{brakken}).  
For closed sets as in Definition \ref{openset}, realisers for this choice principle would exhibit behaviour similar to $\beta^{3}$ as in Theorem \ref{reverand}.  

\smallskip

We also have the following, based on a consequence of the Lindel\"of lemma. 
\begin{cor}
The system $\Z_{2}^{\omega}$ cannot prove the following statement: \emph{a locally bounded function on $\R$ is dominated by a continuous function on $\R$.} 
\end{cor}
\begin{proof}
Note that the statement from the corollary implies item \eqref{mikkel} from Theorem~\ref{noway}, when combined with $\WKL$.  
Theorem \ref{goodd} finishes the proof. 
\end{proof}

\subsection{Heine-Borel theorem for non-basic coverings}\label{storng}
We study the Heine-Borel theorem for countable coverings consisting of arbitrary open sets.  
These results are both complimentary (see Theorem \ref{goodd2}) and establish the robustness of the previous results (see Theorem \ref{goodd3}).

\smallskip

First of all, we define a `strong' version of $\HBC$ as follows. 
\bdefi[$\HBC_{\s}$]
Let $C\subseteq [0,1]$ be closed and let $(O_{n})_{n\in \N}$ be a sequence of open sets with $C\subseteq \cup_{n\in \N}O_{n}$.  Then $C\subseteq \cup_{n\leq n_{0}}O_{n}$ for some $n_{0}\in \N$.
\edefi
We note that coverings by open sets (rather than just intervals) are also studied in e.g.\ \cite{simpson2}*{IV.1.5} and \cite{brownphd}*{Ch.\ III}. Also note that adding the complement of $C$ to the covering, we get a covering of $[0,1]$, so we may as well restrict the definition to $C = [0,1]$.
The strong version $\HBC_{\s}$ is not provable from third-order comprehension.
\begin{thm}\label{goodd2}
The system $\Z_{2}^{\omega}$ cannot prove $\HBC_{\s}$ for $C=[0,1]$.
\end{thm}
\begin{proof}
By Theorem \ref{goodd} since $\HBC$ is a special case of $\HBC_{\s}$ for $C = [0,1]$.
\end{proof}
Despite its larger scope, $\HBC_{\s}$ is not that different from $\HBC$, as follows.  
\begin{thm}\label{goodd3}
The system $\RCAo+\QFAC^{0,1}$ proves $\WKL\asa \HBC_{\s}$.
\end{thm}
\begin{proof}
The reverse direction follows from Corollary \ref{simpleright}.
The forward implication follows from Theorem \ref{dich} with slight modification as follows.  
In case $(\exists^{2})$, one uses the same `proof by contradiction' as in the proof of Theorem \ref{dich}, starting with $(\forall m^{0})(\exists x\in C)\big[x\not \in  \cup_{n\leq m}O_{n}\big]$ and using $\exists^{2}$ to decide $x\in O_{n}$ whenever necessary. 
In case $\neg(\exists^{2})$, we consider \cite{kohlenbach4}*{Prop.\ 4.4 and 4.10}.  These theorems establishes that in $\RCAo+\WKL$, one can prove that a function $Y^{2}$ that is continuous on $2^{\N}$, also has an RM-code on $2^{\N}$.
Careful inspection of the proofs shows that (i) it also applies to other compact spaces like $[0,1]$ and (ii) one can also prove that a \emph{sequence} of continuous functions has a \emph{sequence} of RM-codes.
Hence, we may assume $O_{n}$ is given by an RM-code $\alpha_{n}$ for a continuous function on $[0,1]$.  By \cite{simpson2}*{II.7.1}, such an $\alpha_{n}$ (effectively) gives rise to RM-codes for open sets.  
In this way, $\HBC_{\s}$ reduces to its second-order version, and the latter is provable in $\WKL_{0}$ by \cite{simpson2}*{IV.1.5}.
\end{proof}

\section{Coding open sets}\label{limp3}
We study the questions raised in Section \ref{rmintro} concerning representations of open sets.  
In Section \ref{rmlimp3}, we study the coding principle $\open$ expressing that every characteristic function of an open set has an RM-code.  
We show that $\Z_{2}^{\Omega}$ proves $\open$, while $\Z_{2}^{\omega}$ cannot.  We also show that $\SIXK+\open$ proves (second-order) transfinite recursion for $\Pi_{k}^{1}$-formulas, where the latter is intermediate between $\SIXk$ and $\Pi_{k+1}^{1}$-$\textsf{CA}_{0}$.
We show that various theorems pertaining to open sets, like the Cantor-Bendixson or perfect set theorem, imply $\open$ or the associated comprehension principle $\BOOT^{-}$.
In Section \ref{ctlimp3}, we study how hard it is to compute a code as in $\open$ in terms of the other data.  We show that no type two functional suffices, while $\exists^{3}$ does; the latter however yields full $\Z_{2}$.  
In Section~\ref{waycool}, we shall study similar results for finer representations (compared to Definition \ref{openset}).

\subsection{Reverse Mathematics}\label{rmlimp3}
We study the following theorem pertaining to the countable representation of open sets.  
Let $(q_{n},r_{n})_{n\in \N}$ be a fixed enumeration of all non-trivial open balls $B(q_{n}, r_{n})$ with rational center and radius.
Note that we use `open set' in the sense of Definition \ref{openset}.
\bdefi[$\open$]
For every open set  $Y\subseteq \R$, there is $X\subset \N$ such that $(\forall n\in \N)(n\in X\asa B(q_{n}, r_{n})\subseteq Y)$.
\edefi
Note that for $X, Y$ as in $\open$, we have $x\in Y\asa (\exists m\in \N)(m\in X\wedge x\in B(q_{m}, r_{m}))$, i.e.\ $\open$ endows open sets in $\R$ with a countable representation.
According to Bourbaki (\cite{bourken}*{p.\ 222}), Cantor first proved that open sets can be written as countable unions of open intervals, i.e.\ $\open$ also carries some historical interest. 
However, Aczel states in \cite{as}*{p.\ 134} that constructive set theory cannot prove $\open$.  %, according to

\smallskip

We need the following comprehension principle, which is a special case of the comprehension principle $\BOOT$ from Section \ref{HCT} and is inspired by \cite{simpson2}*{V.5.2}.  
\bdefi[$\BOOT^{-}$]
For $Y^{2}$ such that $(\forall n\in \N)(\exists \textup{ at most one } f^{1})(Y(f, n)=0)$, we have $(\exists X\subset \N)\big(\forall n\in \N)(n\in X\asa (\exists f^{1})(Y(f, n)=0)\big)$.
\edefi
The name of the previous principle is derived from the verb `to bootstrap', as combining the relatively weak\footnote{Note that $\ACAo$ is conservative over $\ACA_{0}$ by \cite{hunterphd}*{Cor.\ 2.5}, while the $\ECF$-translation of $\BOOT$ is provable in $\ACA_{0}$, i.e.\ $\RCAo+\BOOT$ has the first-order strength of $\ACA_{0}$.} 
(in isolation) principles $(\exists^{2})$ and $\BOOT^{-}$, gives rise to the much stronger principle of arithmetical transfinite recursion, as in Theorem~\ref{dufje}.   A more general result is proved in Corollary \ref{col} below. 
\begin{thm}\label{dufje}
The system $\ACAo+\BOOT^{-}$ implies $\ATR_{0}$.
\end{thm}
\begin{proof}
We recall \cite{simpson2}*{V.5.2} which states that $\ATR_{0}$ is equivalent to the following second-order version of $\BOOT^{-}$: for every arithmetical $\varphi(n,X)$ 
such that $(\forall n\in \N)(\exists \textup{ at most one } X\subseteq\N)\varphi(n, X)$, there is $Z\subset \N$ such that 
\be\label{luk}
(\forall n\in \N)(n\in Z\asa (\exists X\subset \N)\varphi(n, X)).  
\ee
To prove the latter principle, we consider the Kleene normal form lemma as in \cite{simpson2}*{V.5.4} which expresses that an arithmetical formula $\psi(X)$ is equivalent to $(\exists f^{1})(\forall n^{0})\theta_{0}(\overline{X}n, \overline{f}n)$, where the formula $\theta_{0}$ is bounded and we also have that $(\forall X)(\exists \textup{ at most one } f^{1})(\forall n^{0})\theta_{0}(\overline{X}n, \overline{f}n)$.  In this light, for arithmetical $\varphi(n, X)$, we have that $(\forall n\in \N)(\exists \textup{ at most one } X^{1})\varphi(n, X)$ is equivalent to the formula
$(\forall n\in \N)(\exists \textup{ at most one } f^{1})(Y(f, n)=0)$ for some $Y^{2}$ defined in terms of $\exists^{2}$.  Applying $\BOOT^{-}$ then yields the required set $Z$ as in \eqref{luk}.  
\end{proof}
The importance of $\BOOT^{-}$ is illustrated by Theorem \ref{cromagnon}.  We also need the following principle from \cite{troeleke1}, which is used in \cite{dagsamIII}*{\S3} to derive e.g.\ $\HBU$.
\bdefi[$\NFP$]\label{FTP}
For any $\Pi_{\infty}^{1}$-formula $A$ with any parameter:
\[
(\forall f^{1})(\exists n^{0})A(\overline{f}n)\di (\exists \gamma^{1}\in K_{0})(\forall f^{1})A(\overline{f}\gamma(f)).
\]
\edefi
Here, `$\gamma^{1}\in K_{0}$' expresses that $\gamma^{1}$ is an \emph{associate}, which is the same as a \emph{code} from RM by \cite{kohlenbach4}*{Prop.\ 4.4}.  Formally, `$\gamma^{1}\in K_{0}$' is the following formula:
\[
(\forall f^{1})(\exists n^{0})(\gamma(\overline{f}n)>_{0}0) \wedge (\forall n^{0}, m^{0},f^{1}, )(m>n \wedge \gamma(\overline{f}n)>0\di  \gamma(\overline{f}n)=_{0} \gamma(\overline{f}m) ).
\]
The value $\gamma(f)$ for $\gamma\in K_{0}$ is defined as the unique $\gamma(\overline{f}n)-1$ for $n$ large enough.  

\smallskip

We now have the following theorem, where $\BOOT$ was introduced in Section \ref{HCT}.
\begin{thm}\label{cromagnon}
The system $\RCAo$ proves $ \BOOT\di [\open+\ACA_{0}]\di \BOOT^{-}$ and $\RCAo+\INDD^{\omega}$ proves $\NFP\di \BOOT\di \HBU_{\closed}$.
\end{thm}
\begin{proof}
The implication $\NFP\di \BOOT$ follows from the proof of \cite{samnetspilot}*{Theorem 3.5} combined with \cite{samph}*{Theorem 3.6}.  
A sketch is as follows: assume that $\BOOT$ is false for some $Y_{0}$.  The formula expressing this has the form $(\forall X\subset \N)(\exists n\in \N)A(X, n)$, 
and a trivial modification yields $(\forall X\subset \N)(\exists n\in \N)B(\overline{X}n)$, i.e.\ only the first $n$ digits of $X$ are used.  
Let $\gamma\in K_{0}$ be as provided by $\NFP$ and use $\WKL$ to obtain an upper bound on $C$ on $\gamma$. 
However, $\INDD^{\omega}$ proves `bounded comprehension' (see \cite{simpson2}*{II.3.9}) for arbitrary formulas, yielding a contradiction. 

\smallskip

For $[\open +\ACA_{0}]\di \BOOT^{-}$, in case $\neg(\exists^{2})$, all functions on Baire space are continuous by \cite{kohlenbach2}*{\S3}. 
Thus, $(\exists f^{1})(Y(f, n)=0)$ is equivalent to $(\exists \sigma^{0^{*}})(Y(\sigma*{00}\dots, n)=0)$, and $\ACA_{0}$ provides the set required by $\BOOT^{-}$. 
In case $(\exists^{2})$, let $Y^{2}$ satisfy $(\forall n\in \N)(\exists \textup{ at most one } f^{1})(Y(f, n)=0)$.  
The formula $(\exists f^{1})(Y(f, n)=0)$ is equivalent to $(\exists X \subset \N^{2})(Y(F(X), n)=0)$, where $F(X)(n):=(\mu m)((n,m)\in X)$.
Hence, $(\exists f^{1})(Y(f, n)=0)$ is equivalent to a formula $(\exists f\in 2^{\N})(\tilde{Y}(f, n)=0)$, where $\tilde{Y}$ is defined explicitly in terms of $Y$ and $\mu^{2}$. 
Now define $Z:\R\di \R$ as:  
\be\label{honker}
Z(x):=
\begin{cases}
0 &  n<_{\R}|x|\leq_{\R} n+1 \wedge \tilde{Y}(\eta(x)(0),n)\times \tilde{Y}(\eta(x)(1),n)=0\\
1 & \textup{otherwise}
\end{cases}, 
\ee
where $\eta(x)$ provides a pair consisting of the binary expansions of $x-\lfloor x \rfloor$; the pair consists of identical elements if there is a unique such expansion.  
Note that $\exists^{2}$ can define such functionals $Z$ and $\eta^{1\di (1\times 1)}$.
By definition, for each $n\in \N$ there is at most one real  $y\in (n, n+1]$ such that $Z(y)=0$. 
Hence,  $Z$ is open and we may apply $\open$ to obtain $X\subset \N$ such that $(\forall n\in \N)(n\in X\asa B(q_{n}, r_{n})\subset Z)$. 
Now note that for any $m\in \N$, the left-hand side of the following formula is decidable:
\[\textstyle
\overline{B}(m+\frac{1}{2}, \frac{1}{2})\subset Z \asa (\forall f^{1})(Y(f, m)>0), 
\] 
which is sufficient to obtain $\BOOT^{-}$ in this case. The law of excluded middle now finishes this part of the proof.
For the implication $\BOOT\di \open$, we use $(\exists^{2})\vee \neg(\exists^{2})$ as follows: in the former case $\BOOT$ readily yields the set $X$ as in $\open$, while $\ACA_{0}$ does the job in the latter case as quantifiers over the reals may be replaced by quantifiers over the rationals if all functions on $\R$ are continuous (as they are in this case by \cite{kohlenbach2}*{\S3}).

\smallskip

For the final implication, $\BOOT\di \open$ means that $\HBC$ reduces to $\HBC_{\RM}$, and the latter is equivalent to $\WKL_{0}$ by \cite{brownphd}*{Lemma 3.13}.
We therefore have access to $\HBC$, as well as $(\exists^{2})$ in the same way as in the previous paragraphs.   
Now let $C\subseteq [0,1]$ be closed and fix $\Psi:I\di \R^{+}$.  Use $\BOOT$ to define $(a_{n}, b_{n})$ as $B(q_{n}, r_{n})$ in case $(\exists x\in C)(x\in B(q_{n}, r_{n})\subseteq I_{x}^{\Psi})$, and $\emptyset$ otherwise.  
Then clearly $\cup_{n\in \N}(a_{n}, b_{n})$ covers $C$, and $\HBC$ yields a finite sub-covering.  Now use $\INDD^{\omega}$ to conclude that this finite sub-covering yields a finite sub-covering for $\cup_{x\in C}I_{x}^{\Psi}$ `by definition'.
\end{proof}
In the following corollary, the second-order system $\Pi_{k}^{1}\mTR_{0}$ is transfinite recursion for $\Pi_{k}^{1}$-formulas, which lies between $\SIXk$ and $\Pi_{k+1}^{1}$-$\textsf{CA}_{0}$ (see \cite{simpson2}*{Table 4}).

\begin{cor}\label{col}
The system $\SIXK+\open$ implies $\Pi_{k}^{1}\mTR_{0}$.
\end{cor}
\begin{proof}
The case $k=0$ follows from the theorem and Theorem \ref{dufje}.  The general case follows in the same way: one simply observes that the proof of `$2\di 1$' from \cite{simpson2}*{V.5.2} relativises to $\SS_{k}^{2}$ for $k^{0}>0$.
\end{proof}
Thus, we have established that $\open$ is hard to prove and moreover is `explosive' as it becomes much stronger when combined with (higher-order) comprehension. 

\smallskip

Finally, we show that $\BOOT^{-}$ shows up in the study of the Cantor-Bendixson theorem and located sets.  % both involving closed sets as in Definition \ref{openset}. 
Now, the former theorem states that a closed set can be expressed as the union of a perfect closed set and a countable set;
this theorem for RM-closed sets is equivalent to $\FIVE$ (\cite{simpson2}*{VI.1.6}).  %We define a version based on Definition \ref{openset}.  
\begin{princ}[$\CBT$]
For any closed set $C\subseteq \R$, there exist $P, S\subset C$ such that $C=P\cup S$, $P$ is perfect and closed, and $S^{0\di 1}$ is a sequence of reals.
\end{princ}
To be absolutely clear, the countable set $S$ is given as a sequence of real numbers, just like in second-order RM.    
Furthermore, a set $C$ in a metric space is \emph{located} if $d(x, C)=\inf_{y\in C}d(x, y)$ exists as a continuous function. 
By \cite{withgusto}*{Theorem 1.2}, $\ACA_{0}$ is equivalent to the locatedness of non-empty RM-closed sets in the unit interval.  Similar results hold for the suprema of RM-closed sets.  % in $[0,1]$.
\begin{princ}[$\CLO$]
Any non-empty closed set $C\subseteq \R$ is located. 
\end{princ}
By \cite{kohlenbach4}*{\S4}, a continuous function on Baire space (and the same for $\R$) has a RM-code in $\RCAo+\ACA_{0}$, i.e.\ the notion of continuity in $\CLO$ does not matter.

\smallskip 

Let $\QFAC^{0,1}_{!}$ be $\QFAC^{0,1}$ restricted to $Y$ such that $(\forall n^{0})(\exists! f^{1})(Y(f, n)=0)$, i.e.\ \emph{unique} existence.
We have the following theorem.
\begin{thm}[$\RCAo$]\label{croft}
Either $\CLO$ or $\CBT$ implies $\BOOT^{-}+\QFAC^{0,1}_{!}$.
\end{thm}
\begin{proof}
The proof is trivial in case $\neg(\exists^{2})$, in the same way as in the proof of Theorem~\ref{cromagnon}.  
Thus, we may assume $(\exists^{2})$.
Let $Y^{2}$ be as in $\BOOT^{-}$ and let $C$ be the complement of the open set $Z$ defined in \eqref{honker} in the proof of Theorem \ref{cromagnon}.
By definition, for each $n\in \N$ there is at most one real  $y\in (n, n+1]$ such that $Z(y)=0$.  To obtain $\BOOT^{-}$, we proceed as follows.    

\smallskip

First assume $\CLO$ and note that to check whether $(\exists f^{1})(Y(f,n)=0)$, it suffices to: (i) check $Z(n+\frac{1}{2})=0$, (ii) if $Z(n+\frac{1}{2})\ne 0$, then consider $d(n+\frac{1}{2}, C)$; the latter is inside $(n, n+1]$ if and only if $(\exists f^{1})(Y(f,n)=0)$.  Inequalities of real numbers can be decided by $\exists^{2}$, and $\BOOT^{-}$ readily follows.  

\smallskip
  
Secondly, assume $\CBT$ and note that the set $C$ by assumption consists of isolated points (only).  Also, $(\exists f^{1})(Y(f,n)=0)$ is equivalent to a point of $C$ being in $(n, n+1]$.  
Now, for the sets $P, S$ provided by $\CBT$, we must have $P=\emptyset$ and $S$ lists all points of $C$.  Thus, $\BOOT^{-}$ readily follows using $\exists^{2}$, and we are done.
  
\smallskip

Thirdly, to prove $\QFAC^{0,1}_{!}$, assume $(\forall n^{0})(\exists! f^{1})(Y(f, n)=0)$ and again consider the aforementioned set $C$.  
By assumption, for every $n$ there is exactly one $y\in \R$ such that $y\in (n, n+1]\cap C$.  In the same way as in the previous two paragraphs, 
the set $S$ from $\CBT$ and the function $d(x, C)$ from $\CLO$ allow one to find this unique real.  The theorem now follows.  
\end{proof}
As a corollary, we show that the perfect set theorem for our notion of closed set as in Definition \ref{openset} also gives rise to $\BOOT^{-}$.  
The second-order version of the perfect set theorem is equivalent to $\ATR_{0}$ by \cite{simpson2}*{V.5.5}.
Note that a set is \emph{uncountable} in RM if there is no sequence that lists its elements (see e.g.\ \cite{simpson2}*{p.\ 193}).
\begin{princ}[$\PST$]
For any closed and uncountable set $C\subseteq [0,1]$, there exist $P\subseteq C$ such that $P$ is perfect and closed.  
\end{princ}
\begin{cor}
The system $\RCAo$ proves $\PST\di \BOOT^{-}$.
\end{cor}
\begin{proof}
The set $C$ from the proof of the theorem does not have perfect subsets, and hence $\PST$ provides a sequence that lists the elements of $C$.  
The proof of the theorem now yields $\BOOT^{-}$.
\end{proof}
We have shown that the Cantor-Bendixson theorem, the perfect set theorem, and located sets give rise to $\BOOT^{-}$ when using closed sets as in Definition \ref{openset}.  
Hence, a slight change to the aforementioned theorems makes them much harder to prove by the above.  Similar results no doubt exist for other theorems pertaining to closed sets from RM.  
We do not know whether $\BOOT^{-}$ gives rise to (nice) equivalences, but do have Theorem \ref{blind} to offer in this regard.    
\begin{thm} \label{blind}
Over $\RCAo$, we have the following results.
\begin{enumerate}
 \renewcommand{\theenumi}{\alph{enumi}}
\item $(\exists^{3})\di \CLO \di \HBC$.\label{hier}
\item $\CLO\asa [\ACA_{0}+ \open]$.\label{CLO}
\end{enumerate}
\end{thm}
\begin{proof}
Item \eqref{hier} is immediate from item \eqref{CLO} and the usual interval halving technique using $\exists^{3}$; alternatively, note that `$x\in C$' can be replaced by $d(x, C)=_{\R}0$ and use the proof of Theorem \ref{dich} with $\QFAC^{0,1}$ omitted.
For the reverse direction in item \eqref{CLO}, \cite{withgusto}*{Theorem 1.2} implies that $\ACA_{0}$ is equivalent to the locatedness of non-empty RM-closed sets in the unit interval. 
This also establishes $\CLO\di \ACA_{0}$ as the antecedent applies to RM-closed sets, while a code for $d(x, C)$ is not needed in the proof. 
For the forward direction in item~\eqref{CLO}, continuous functions on $\N^{\N}$ (and similar for $\R$) have RM-codes given $\ACA_{0}$ by \cite{kohlenbach4}*{\S4}.
Hence, for an open set $O$, we have $x\in O\asa d(x, O^{\c})>_{\R}0$ where the latter is a code for the distance function provided by $\CLO$.  
However, the right-hand side of this equivalence defines an RM-open set by \cite{simpson2}*{II.7.1}, i.e.\ $\open$ follows immediately. 
\end{proof}
The notion of \emph{weakly located} set, studied in e.g.\ \cite{withgusto}, gives rise to a version of item~\eqref{CLO} at the level of $\WKL_{0}$.
We can obtain the same results as in Theorem \ref{blind} for $\CLO$ replaced by `for a sequence of closed sets $C_{n}\subset [0,1]$, there is a sequence $(x_{n})_{n\in \N}$ such that $\sup C_{n}= x_{n}$'.
We can also establish results like $[\ACA_{0}+\CBT']\asa [\FIVE+\open]$ over $\RCAo$, where $\CBT'$ is $\CBT$ with the `extra data' that $P$ is located in $C=P\cup S$. 
The same holds for $\PST'\asa [\ATR_{0}+\open]$ while the extra `locatedness' data does seem necessary, i.e.\ no equivalence is possible without.

\smallskip

The point of the previous results is twofold:  on one hand, Theorem \ref{blind} shows that locating closed sets as in Definition \ref{openset} is hard, while this operation can be done without the Axiom of Choice. 
On the other hand, the previous equivalences reduce to known second-order results under $\ECF$ from Remark \ref{ECF}, as open sets as in Definition \ref{openset} are translated to open sets as in RM by \cite{simpson2}*{II.7.1}.  
In other words, second-order RM can be seen as a reflection of higher-order RM under a lossy translation.
This observation constitutes the main motivation for the results in \cite{samph}, where the comparison is 
made with Plato's ideas on ideal objects.     

\smallskip

Finally, a different notion of open set, namely given by uncountable unions of open balls, yields nice equivalences involving $\BOOT$, 
Cantor-Bendixson theorem, perfect set theorem, and located sets, as explored in \cite{samph}.  
Moreover, the results pertaining to $\CBT$ and $\PST$ are nice, but seem to depend on the particular notion of `(un)countable' used.  
As also shown in \cite{samph}, this problem does not occur for open sets given by uncountable unions of open balls, i.e.\ the usual definition of `countable' can be used.

\subsection{Computability theory}\label{ctlimp3}
In this section, we study the computational properties of $\open$, i.e.\ how hard is it to compute (Kleene S1-S9) the representation provided by this coding principle?
Analogous results are obtained in Section \ref{waycool} for `finer' representations of open sets. 

\smallskip

First of all, we introduce two functionals based on Definition \ref{openset}, as follows. 
As for $\Theta$-functionals, we simplify the associated type to `$3$'.
\bdefi[Open sets and computability]\label{fuha}~
\begin{enumerate}
\item A $\Phi$-functional is any $\Phi^{3}$ such that for every open set $Y \subset \R$ and real $x \in Y$, $\Phi(Y, x)$ equals a rational $r^0>0$ such that $B(x, r) \subset Y$.
\item A $\Psi$-functional is any $\Psi^{3}$ that for every open set $Y \subset \R$, $\Psi(Y)$ equals two sequences of rationals $(a_n)_{n\in \N}, (r_n)_{n\in \N}$ such that $Y = \cup_{n\in \N} B(a_n, r_n)$.  
\end{enumerate}
\edefi
It goes without saying that these functionals are not unique.  
We allow $r_{n}=0$ and note that those instances can be removed using $\mu^{2}$ in case $Y\ne \emptyset$.  Note that one can test for the latter condition for open sets using $\mu^{2}$. 

\smallskip

Note that $\Phi$-functionals are computable in $\mu^{2}$ and $\exists^{3}$, while the same holds for $\Psi$-functionals by Theorem \ref{bartles}; a better result cannot be expected by the following. 
\begin{thm}\label{jawadde}
No $\Phi$-functional is countably based.  
\end{thm}
\begin{proof} 
Assume that $\Phi(Y,0) = r$, where $0 \in Y$, and where $Y$ is the constant 1, i.e.\ it represents $\R$. 
If $\Phi$ is countably based, there is a countable set $X$ such that if $Y'$ agrees with $Y$ on $X$ and defines an open set, then $\Phi(Y,0) = \Phi(Y',0)$. 
However, there will be a real $r' \not \in X$ such that $0 < r' < r$. Then consider $Y' := Y \setminus \{r'\}$. 
Since $r$ is an unacceptable value for $\Phi(Y',0)$, we obtain a contradiction.
\end{proof}
\begin{cor}
No $\Psi$-functional is countably based.  
\end{cor}
\begin{proof}
Note that any $\Psi$-functional computes a $\Phi$-functional modulo $\mu^2$: the latter can be used to find the right open ball in the sequence provided by $\Psi$.  
The class of countably based functionals is closed under Kleene computability (see \cite{hartleycountable}).
\end{proof}
Recall from Section \ref{HCT} that realisers for $\HBU_{\c}$ are called $\Theta$-functionals and compute the finite sub-covering in $\HBU_{\c}$ in terms of the other data.
We similarly use $\Theta_{\closed}$ to denote a realiser for $\HBU_{\closed}$ from Section \ref{rmopen}. 
\begin{thm}
A realiser $\Theta_{\closed}$ for $\HBU_{\closed}$ together with $\mu^{2}$ computes a $\Psi$-functional.
\end{thm}
\begin{proof}
Fix some open set $U\subset [0,1]$ and let $C$ be its complement in $[0,1]$.
Define $\Psi_{n}(x)$ as $\frac{1}{2^{n}}$ if $x\in C$, and $0$ otherwise.  Clearly, for fixed $n^{0}$, $\lambda x.\Psi_{n}(x)$ yields a covering of $C$, implying
\be\label{doubleup}
(\forall n^{0})(\exists y_{1}, \dots, y_{k}\in C)(\forall x\in C)(x\in \cup_{i\leq k}I_{y_{i}}^{\Psi_{n}}), 
\ee
where $w_{n}:=\Theta_{\closed}(\lambda x.\Psi_{n}(x))$ provides the finite sequence of $y_{i}$'s.  
With minor modification, $[0,1]\setminus \cup_{n\in \N}\big( \cup_{i<|w_{n}|}I_{w_{n}(i)}^{\Psi_{n}}\big)$ yields a countable union of open intervals $(a_{n}, b_{n})$ such that $U=\cup_{n\in \N}(a_{n}, b_{n})$.
\end{proof}
Finally, the `standard' proof that an open set in $\R$ is the union of countable many open balls, goes through modulo $(\exists^{3})$.
\begin{thm}\label{bartles}
The system $\Z_{2}^{\Omega}$ proves $\open$.  
A $\Psi$-functional can be computed from $\exists^{3}$ via a term from G\"odel's $T$.
\end{thm}
\begin{proof}
Let $Y\subset \R$ be open and define for $x\in Y$ the (non-empty by definition) sets $A_{x}:=\{a\in \R : (a,x] \subset Y\}$ and  $B_{y}:=\{b\in \R : [b, y) \subset Y\}$ using $\exists^{3}$, letting them be the empty set if $x\not\in Y$.  
If the set $A_{x}$ (resp.\ $B_{x}$) has a lower (resp.\ upper) bound (which is decidable assuming $\exists^{3}$), then $a_{x}:=\inf A_{x}$ (resp.\ $b_{x}:=\sup B_{x}$) exists thanks to $\exists^{3}$, using the usual interval-halving technique.  
In case such a bound is missing, we use a default value for $a_{x}$ (resp.\ $b_{x}$) meant to represent $-\infty$ (resp.\ $+\infty$).
We define $J_{x}:=(a_{x}, b_{x})$ and will show that $Y=\cup_{q\in \Q} J_{q}$, thus establishing the theorem.  By the definition of $J_{x}$, we must have $J_{x}\subset Y$ for all $x \in Y$. 
Since also $Y\subset \cup_{x\in Y}J_{x}$ due to $x\in J_{x}$ if $x\in Y$, we actually have $Y=\cup_{x\in Y}J_{x}$, and note that $\exists^{3}$ guarantees that this union actually exists.
One readily shows that either $J_{x}=J_{y}$ or $J_{x}\cap J_{y}=\emptyset$ for $x, y \in Y$.  Hence, for $x\in Y$, $J_{x}=J_{q}$ for any $q\in J_{x}\cap \Q$, and the theorem follows. 
%Suppose $z\in J_{x}\cap J_{y}$ for $x, y\in Y$ (implying $x\ne_{\R} y$).  Then if $x<y$, we have $a_{y}<z <b_{x}$, and if $y<x$, we have $a_{x}<z<b_{y}$.  % depending on how the intervals overlap.  
%In the first case, we have for $a_{x}<w<x$ that $ (w, y]\subset Y$, as $(w, y]=  (w, x]\cup [x, z]\cup  [z,y]$ and the latter three intervals are in $Y$ by assumption.  However, this implies that $a_{x}=a_{y}$, and $b_{x}=b_{y}$ follows in the same way.  Hence, $J_{x}=J_{y}$ if $x<y$, and the other case is treated in the same way.  
Note that the second part follows immediately. 
\end{proof}

\section{The Urysohn and Tietze theorems}\label{urytiet}
We establish the results sketched in Section \ref{rmintro} for the Urysohn and Tietze theorems;
the latter are basic results of topology that are well-known in RM:  for RM-codes of closed sets, these theorems even hold recursively, i.e.\ the objects 
claimed to exists may be computed (in the sense of Turing) from the data by \cite{simpson2}*{II.7}.  We now show that the situation is dramatically different for our notion of closed sets. 
In particular, we connect these theorems to the coding principle $\open$ from the previous section and establish they suffer from the Pincherle phenomenon. 

\smallskip

First of all, we study the Urysohn lemma for $\R$ and Definition \ref{openset} as follows.
\bdefi[$\URY$]
For closed disjoint sets $C_{0}, C_{1}\subseteq \R$, there is a continuous function $g:\R\di [0,1]$ such that $x\in C_{i}\asa g(x)=i$ for any $x\in \R$ and $i\in \{0,1\}$.
\edefi
\bdefi
A realiser for $\URY$ is called a \emph{$\zeta$-functional}, i.e.\ for disjoint closed sets $C_{0},C_{1}\subset \R$, $\zeta(C_{0}, C_{1})$ is a continuous function such that $\zeta(C_{0}, C_{1})(x)=i$ whenever $x\in C_{i}$ for $i=0,1$.
\edefi
In case one works with \emph{codes} rather than actual sets, the Urysohn lemma is `recursively true' by \cite{simpson2}*{II.7.3}, i.e.\ one can \emph{compute} (in the sense of Turing) a code for the continuous function from a code for the sets $C_{i}$, over $\RCA_{0}$.
How different things are at the higher-order level: we have the following RM and relative computability result.  Recall the notion of  $\Phi$-functional from Definition \ref{fuha}.
\begin{thm}\label{yesway}
A $\Phi$-functional can be computed from a $\zeta$-functional and $\mu^{2}$.  A $\zeta$-functional can be computed from a $\Psi$-functional and $\mu^{2}$; $\ACAo$ proves $\URY\asa \open$. 
\end{thm}
\begin{proof}
For the first part, fix an open set $U\subset \R$, define disjoint closed sets $C_{0}=U^{\c}$ and $C_{1}=\emptyset$, yielding $x\in U \asa \zeta(U^{\c}, \emptyset)(x)>_{\R}0$ for all $x\in \R$.  
Since $\lambda x.\zeta(U^{\c}, \emptyset)$ is continuous, $\mu^{2}$ provides a modulus of (pointwise) continuity by (the proof of) \cite{kohlenbach4}*{Prop.\ 4.7}.  Using this modulus (and $\mu^{2}$), we
may find $r>0$ such that for all $y\in B(x, r)$, we have $\zeta(U^{\c}, \emptyset)(y)>0$, which is exactly as required for a $\Phi$-functional.  Note that $\URY\di \open$ also follows, over $\ACAo$.  
The second part involving $\Psi$-functionals follows by definition.

\smallskip

For the implication $\open\di \URY$ over $\ACAo$, define $h(x)$ as $i$ for $x\in C_{i}$ and $i=0,1$.  In case $x\in Z= C_{0}^{\c}\cap C_{1}^{\c}$, we define $h(x)$ as follows: first note that $Z$ is open (by definition), and hence $Z=\cup_{n\in \N}B(a_{n}, r_{n})$ by $\open$.
Note that $Z$ cannot be empty by assumption.  
Using $\mu^{2}$, we find $m^{0}$ such that $x\in B({a_{m}, r_{m}})$ and we may test if $a_{m}\pm r_{m}$ belong to $C_{0}$ or $C_{1}$.  If such a point is in neither, it is in $Z$, and hence in $B(a_{k}, r_{k}) $ for some  $k\ne m$, and we can repeat the previous process.
There are three possible outcomes:  
\begin{enumerate}
 \renewcommand{\theenumi}{\alph{enumi}}
\item This procedure ends after finitely many steps, say with $a_{k_{0}}-r_{k_{0}}< x < a_{k_{1}}+r_{k_{1}}$ and the former (resp.\ latter) rational is in $C_{0}$ (resp.\ $C_{1}$). \label{awan}
\item This procedure only ends after finitely many steps \emph{in one direction}, say with $a_{k_{0}}-r_{k_{0}}< x$ and this rational is in $C_{0}$.  \label{dan}
\item This procedure never ends in both directions.  \label{tsju}
\end{enumerate}
If item \eqref{tsju} is the case, then $C_{0}=C_{1}=\emptyset$, and $h(x)=0$ everywhere.  If item \eqref{awan} is the case, define $h(x)$ as the (increasing) straight line connecting $(a_{k_{0}}-r_{k_{0}}, 0)$ and $(a_{k_{1}}+r_{k_{1}}, 1)$ for $a_{k_{0}}-r_{k_{0}}< x < a_{k_{1}}+r_{k_{1}}$.
If $C_{1}$ (resp.\ $C_{0}$) is eventually met on the left (resp.\ right), the modification to $h$ is obvious.  
If item \eqref{dan} is the case, then define $h(x):=0$ for $a_{k_{0}}-r_{k_{0}}<x$.  If $C_{1}$ is eventually met on the left, or if the unbounded area is on the left, the modification to $h$ is obvious.  

\smallskip

Finally, an indirect (but over $\RCAo$) proof of $\open\di \URY$ follows from replacing the closed sets 
in the latter by RM-codes and applying the RM-version of the Urysohn lemma, namely \cite{simpson2}*{II.7.3}.  
\end{proof}
We can weaken the base theory in the previous theorem as follows. 
Let $\cont$ be the statement that every continuous $Y:\R\di \R$ has an RM-code, as studied in \cite{kohlenbach4}*{\S4} for Baire space.  
Note that the $\ECF$-interpretation of $\cont$ is a tautology.
\begin{cor}\label{XWX}
The system $\RCAo+\cont$ proves $\URY\asa \open$.
\end{cor}
\begin{proof}
Use $(\exists^{2})\vee \neg(\exists^{2})$ as follows.  In the former case, the theorem provides a proof of $\URY\asa \open$. 
In the latter case, all functions on $\R$ are continuous by \cite{kohlenbach2}*{\S3}, while $\cont$ provides RM-codes.  Thus,
the usual proof of the second-order Urysohn theorem may be used (see \cite{simpson2}*{II.7.3}), while $\open$ follows from \cite{simpson2}*{II.7.1}.
\end{proof}
Secondly, we study the Tietze (extension) theorem, which expresses that a {continuous} function on a closed set can be extended to a continuous function on the whole space, while if the original function is bounded, so is the extended function with the same bound.  Lebesgue (\cite{lebzelf}), de la Vall\'ee-Poussin (\cite{pussy}), and Carath\'eodory (\cite{carapils}) prove special cases of this theorem \emph{not involving boundedness conditions}.  
Furthermore, Tietze states in \cite{tietze}*{p.\ 10} that the given function can be discontinuous outside the closed set and that the boundedness condition may be dropped. 
We therefore consider the following version of the Tietze theorem.
\bdefi[$\TIET$]
For $f:\R\di \R$ continuous on the closed $D\subseteq[0,1]$, there is $g:\R\di \R$, continuous on $[0,1]$ such that $f(x)=_{\R}g(x)$ for $x\in D$.
\edefi
%WRONG: $\HBC$ cannot imply $\URY$ or $\TIET$.  Changed Thm \ref{tieten} and proof accordingly.  
%We have the following RM-results involving a nice implication. 
\begin{thm}\label{tieten}
The system $\ACAo+\QFAC^{0,1}$ proves $\TIET\asa \URY$ and $\RCAo$ proves $[\TIET+\WKL]\di \HBC$.
\end{thm}
\begin{proof}
For the implication $\URY\di \TIET$, we have $\URY\di \open$ by Theorem \ref{yesway} and we may therefore use the RM-proof of the Tietze extension theorem (see \cite{simpson2}*{II.7.5}).  
The latter applies to bounded functions and we can guarantee boundedness using $\WKL$ (using $\HBC_{\RM}$ in particular).  
%WRONG: added a sentence. 
Note that $\ACA_{0}\di\cont$ over $\RCAo$ is proved in \cite{kohlenbach4}*{\S4}, and the same proof goes through relative to closed sets using $(\exists^{2})$. 

\smallskip

For $\TIET\di \URY$, let $C_{i}$ be as in $\URY$ for $i=0,1$ and define $f$ on $C_{2}:= C_{0}\cup C_{1} $ as follows: $f(x)=0$ if $x\in C_{0}$ and $1$ otherwise.  
\emph{If} $f$ is continuous on $C_{2}$, \emph{then} its extension $g$ provided by $\TIET$ is as required for $\URY$.  To show that 
$f$ is continuous on $C_{2}$, we prove that 
\be\label{sepa}\textstyle
(\forall N^{0})(\exists n^{0})(\forall x\in C_{0}, y\in C_{1})(x, y\in [-N, N]\di |x-y|\geq  \frac{1}{2^{n}}). 
\ee
If \eqref{sepa} is false, $\QFAC^{0,1}$ yields a double sequence $\lambda n.x_{n}, \lambda n. y_{n}$ in $[-N,N]$ such that for all $n^{0}$, we have $x_{n}\in C_{0}, y_{n}\in C_{1}$, and $|x_{n}-y_{n}|<\frac{1}{2^{n}}$.  
As $C_{0}, C_{1}$ are closed and the sequences bounded, there are $x\in C_{0}, y\in C_{1}$ such that $\lim_{n}x_{h_{0}(n)}= x$ and $\lim_{n}y_{h_{1}(n)}=y$ for sub-sequences provided by $h_{0}^{1}, h_{1}^{1}$.  However,  $(\forall n^{0})(|x_{n}-y_{n}|<\frac{1}{2^{n}})$ implies that $x=_{\R}y$, a contradiction since $C_{0}\cap C_{1}=\emptyset$. 
Finally, since \eqref{sepa} provides a positive `distance' between $C_{0}$ and $C_{1}$ in every interval $[-N,N]$, we can always chose a small enough neighbourhood to exclude points from one of the parts of $C$, thus guaranteeing continuity for $f$ everywhere on $C_{2}$. 

\smallskip

For $[\TIET+\WKL]\di \HBC$, let $F$ and $D$ be as in item \eqref{bounk0} of Theorem \ref{noway} and consider the extension $g$ provided by $\TIET$.   
By $\WKL$, $g$ is bounded on $[0,1]$, and $F$ is therefore bounded on the closed set $D$.  Theorem \ref{noway} finishes this implication.
%WRONG
%For the implication $\HBC\di \TIET$, note that the consequent is trivial in case $\neg(\exists^{2})$ as all functions on $\R$ are continuous in this case (see \cite{kohlenbach2}*{\S3}).
%In case $(\exists^{2})$, one obtains a continuous modulus of continuity $h$ for $f$ on $D$ from $\TIET$ using (the proof of) \cite{kohlenbach4}*{Prop.\ 4.10}).  
%This modulus $h$ readily yields a countable cover of $D$ and $\HBC$ provides a finite sub-cover, establishing the uniform continuity of $f$ on $D$.
%In the same way, $f$ has an RM-code that is uniformly continuous on $D$ (and hence bounded there).  The usual second-order proof of the Tietze extension theorem in $\RCA_{0}$ 
%now yields $\TIET$, using $\exists^{2}$ to decide elementhood of open sets whenever necessary.  Note that an RM-code defined on $[0,1]$ readily gives rise to a functional by applying $\QFAC^{1,0}$ to the totality of this code. 
\end{proof}
We improve Theorem \ref{tieten} as follows; the proof is similar to that of Corollary~\ref{XWX}. 
\begin{cor}
The system $\RCAo+\QFAC^{0,1}+\cont$ proves $\TIET\asa \URY$.
\end{cor}
In conclusion, we have established that the Urysohn and Tietze theorems suffer from the Pincherle phenomenon.
Moreover, the previous corollary does not seem to go through over $\RCAo$, unless we introduce moduli of continuity in the theorems at hand.  
Another approach would be to extend the base theory; this option is studied in \cite{samph}*{\S5} for fragments of $\NFP$ from Definition \ref{FTP}.

\section{The Baire category theorem}\label{BCT}
We establish the results sketched in Section \ref{rmintro} for the Baire category theorem.
We introduce realisers for the latter in Section \ref{BCT1} and show that these functionals can be 
computed from $\exists^{2}$ and non-monotone inductive definitions in Section \ref{mainBCT}.  
\subsection{Introduction}\label{BCT1}
The Baire category theorem expresses that a sequence of dense open sets has an intersection that is also dense; this theorem can be found in Baire's 1899 doctoral thesis and Osgood's paper \cite{fosgood}.
This theorem is studied (in some variations) in the computational approaches to mathematics mentioned in Remark~\ref{opensetscoded}. 
It is therefore a natural question what the computational properties of the Baire category theorem are when using Definition \ref{openset}.  
Our main results are:
\begin{enumerate}
\renewcommand{\theenumi}{\alph{enumi}}
\item A realiser $\xi$ for the Baire category theorem (Definition \ref{jawadde2}) can be computed from $\exists^{2}$ and $\IND$, the inductive definition operator (Definition \ref{indje}).\label{more1}
\item There is no Baire-realiser $\xi$ computable in a functional of type 2.\label{more2}
\end{enumerate}
Thus, we observe that the Baire category theorem also (partially) exhibits the Pincherele phenomenon. 
This should be contrasted with the status of the Baire category theorem in e.g.\ the Weihrauch framework (\cite{brakke}).
We also note that the below proofs are quite different from the proofs in \cite{dagsamIII, dagsamV} and the above.

\smallskip

Note that by \cite{heerlijk}*{Theorem 4.102}, $\ZF$ proves that separable completely metrisable spaces have the property of Baire, i.e.\ the Axiom of Choice is not the cause 
of the (computational)  hardness of the Baire category theorem.  The aforementioned notion of realiser is defined as follows.  
\bdefi\label{jawadde2}
A \emph{realiser} for the Baire category theorem is any $\xi^{3}$ such that whenever $ (Y_n)_{n \in \N}$ is a sequence of dense open sets of reals, then 
%\[\textstyle
$\xi(\lambda n.Y_n) \in \bigcap_{n \in\N } Y_n$.
%\]
\edefi
For the previous definition, we assume a standard coding of the reals in $\N^\N$, and that a set $Y \subset \R$ is given in in the form of its characteristic function as in Definition~\ref{openset}. 
So, technically we are working inside the full type-structure over $\N$ up to level 3. All our ``algorithms" are relative to $\exists^2$ and any other objects specified in the argument.

\smallskip

As noted above, the Baire category theorem for separable spaces can be proved in $\ZF$, yielding the existence of a Baire realiser $\xi$; a direct translation of that argument (however) uses computations relative to $\exists^3$. Our first result involves a construction of a Baire realiser that terminates on \emph{all} countable sequences of sets, and gives an element in the intersection when the sets are all open and dense, working within the realm of the \emph{countably based functionals}. 
By contrast, other `basic' theorems about open sets, like \emph{every open set is a union of open intervals with rational endpoints} and \emph{for every closed set $C$, the distance $d(x,C)$ is continuous}, only have realisers that are not countably based.
\bdefi\label{indje}
For $\Gamma:2^\N \rightarrow 2^\N$ satisfying the monotonicity condition $(\forall A\subseteq \N)(\Gamma(A) \supseteq A)$, define the well-ordered sequence of sets $\Gamma_\alpha$ as follows:
\be\label{IT}\textstyle
\Gamma_\alpha := \Gamma\big(\bigcup_{\beta < \alpha}\Gamma_\beta\big).
\ee
For any such $\Gamma$, there is an ordinal $\alpha_{0}$ with $\Gamma_{\alpha_{0} + 1} = \Gamma_{\alpha_{0}}$, where the latter is called the \emph{inductive closure} of $\Gamma$.
Finally, for any $F:2^{\N}\di 2^{\N}$ we define ${\rm \IND}(F)$ as the inductive closure $\Gamma_{\alpha_{0}}$ for the functional $\Gamma$ defined as $\Gamma(A):= A\cup F(A)$.
\edefi
Regarding the previous definition, $\IND$ denotes a type three functional, while $\INDD^{\omega}$ is the induction axiom for all $\L_{\omega}$-functionals. 

\smallskip

Finally, we noted above that the Baire category theorem for separable spaces can be proved in $\ZF$ and this result will be sharpened, for our special case $\BCT$ defined below, as follows. 
\begin{enumerate}
\item[(c)] The system $\RCAo+\BOOT$ proves the Baire category theorem $\BCT$. 
\item[(d)] The system $\Z_{2}^{\omega}$ cannot prove $\BCT$.
\end{enumerate}
Note that $\BOOT\di\open$ by Theorem \ref{cromagnon} and open sets thus have an RM-representation.  
However, $\BCT$ deals with \emph{sequences} of open sets and obtaining the associated sequences of RM-representations would seem to require countable choice.   

\subsection{Main results}\label{mainBCT}
We prove the results sketched in items (a)-(d) in Section \ref{BCT1}. 
\subsubsection{Computability Theory}
We prove the results sketched in items (a) and (b) in Section \ref{BCT1}.  To this end, we first fix some notation, as follows. 
\bdefi[Computational attempts]
\begin{itemize}
\item A \emph{tag} is a pair $(r,\epsilon)$ from $\Q$ where  $\epsilon > 0$. We let $(r,\epsilon)_o = (r - \epsilon , r + \epsilon)$  and $(r,\epsilon)_c$ be the corresponding closed interval $[r - \epsilon , r + \epsilon]$.
\item If $(r_1 , \epsilon_1)$, $(r_2,\epsilon_2)$ are tags, we let $(r_1,\epsilon_1) \prec (r_2 , \epsilon_2)$ if $r_1 = r_2$ and  $\epsilon_1 \geq 2\epsilon_2$.
\item An \emph{attempt} is a sequence ${\bf s} = [(r_1,\epsilon_1) , \ldots , (r_k , \epsilon_k)]$ of tags, where we for $i < k$ have that $(r_{i+1},\epsilon_{i+1})_c \subseteq (r_i,\epsilon_i)_o$ and that $2\epsilon_{i+1} \leq \epsilon_i$.
\item If ${\bf s}$ and ${\bf t}$ are attempts, we let ${\bf s} \lhd {\bf t}$ if either ${\bf s}$ is an initial sub-sequence of ${\bf t}$ or if ${\bf s}$ comes before ${\bf t}$ in the lexicographical ordering on attempts based on `$\prec$' as above for tags.
\end{itemize}
\edefi
The ordering `$\lhd$' is not the Kleene-Brouwer ordering, but a partial ordering all the same.
Intuitively, an attempt will be an attempt to find a shrinking sequence of closed intervals whose single point in the intersection will also be in $\cap_{n\in \N}Y_n$. We have limited access to information about any $Y_n$: we do not  know if it is open and dense. Even when it is, we still have no way of saying that a tag $(r,\epsilon)$ represents a subset of $Y_n$. So certain attempts may lead to failure, and we have to try again with better attempts. This is what will be captured by our inductive definition, 
which will be defined from a given sequence $(Y_n)_{n \in \N}$ in a uniform way, as follows.

\begin{definition}
Let $(Y_n)_{n\in \N}$ be given, and let $A$  be a set of attempts. We define $\Gamma(A)$ in cases as follows, where we order $\Q$ according to a standard enumeration.
\begin{itemize}
\item[o)] If $A$ is not totally ordered by $\lhd$, we let $\Gamma(A) = A$. For the rest of the cases we will assume that $A$ is totally ordered.
\item[i)] In case $A = \emptyset$ and if $Y_1 \cap \Q = \emptyset$, put $\Gamma(A) = A$. If not,  $\Gamma(A) := \{(r_1 , 1)\}$, where $r_1$ is the first rational number in $Y_1$.
\item[ii)] If the set $A$ has a $\lhd$-maximal element $[(r_1,\epsilon_1) , \ldots , (r_k , \epsilon_k)]$ and if $(r_k,\epsilon_k)_o \cap Y_{k+1} \cap \Q = \emptyset$ we let $\Gamma(A) = A$. 
If not, we define
\[
\Gamma(A) = A \cup \{[(r_1,\epsilon_1), \ldots ,(r_k,\epsilon_k),(r_{k+1},\epsilon_{k+1})]\},
\]
where $r_{k+1}$ is the first rational number in $(r_k,\epsilon_k)_o \cap Y_{k+1}$, and then $\epsilon_{k+1}$ is the first positive rational number such that $(r_{k+1},\epsilon_{k+1})_c \subseteq (r_k,\epsilon_k)_o$ and such that $2\epsilon_{k+1} \leq \epsilon_k$
\item[iii)] If there is $k\in \N$ with infinitely many attempts $[(r_1,\epsilon_1) , \ldots , (r_k , \epsilon_k)]$ in $A$, we put $\Gamma(A) = A$.
\item[iv)] If none of the above apply, we proceed as follows: for each $k$, let $[(r_1,\epsilon_1) , \ldots , (r_k , \epsilon_k)]$ be the $\lhd$-maximal attempt in $A$ of length $k$. It is easy to see that for $i \leq k$, the tag $(r_i,\epsilon_i)$ will be independent of the choice of $k$.  Let $x$ be the unique element in $\bigcap_{i \in \N}(r_,\epsilon_i)_c$.  We consider the following sub-cases:
\begin{enumerate}
\item If $x \in \bigcap_{i \in \N}Y_i$, let $\Gamma(A) = A$
\item Otherwise, let $k$ be minimal such that $x \not \in Y_k$ and we define $\Gamma(A) = A \cup \{[(r_1,\epsilon_1) , \ldots , (r_{k-1},\epsilon_{k-1}), (r_k , \frac{\epsilon_k}{2})]\}$.
\end{enumerate}
\end{itemize}
\end{definition}
By construction, in all cases where $A$ is totally ordered by $\lhd$, we either define $\Gamma(A) = A$ or we add one new element on top of $A$. 
This means that $\Gamma$, seen as an inductive definition, will generate a well-ordered set $A_\infty$ of attempts such that $\Gamma(A_\infty) = A_\infty$. 
Moreover, we define $\xi(\lambda n.Y_n) = x$ as in case  iv).(1) if this is the `stopping' case. 
Otherwise, we define $\xi(\lambda n.Y_n) = 0$. 

\smallskip

The resulting functional $\xi$ is clearly computable in $\IND$ and $\exists^2$. It remains to show that iv).(1) will be the `stopping' case when each $Y_n$ is open and dense. 
We consider the other alternatives, and show that they are impossible in this situation.
\begin{itemize}
\item Since we generate a totally ordered set, case o) will never be relevant for any input $(Y_n)_{n \in \N}$.
\item Since $Y_1$ is open and dense,  we start the recursion by adding an element in case i). Since $Y_{k+1}$ is open and dense, we will also continue the recursion when we are in case ii).
\item The remaining alternative is case iii). In this case, there will be a least $k$ for which this is possible. Since the only way to develop $A$ sideways (using $\prec$ on tags) is via case iv).2, there will be a maximal attempt $[(r_1 , \epsilon_1) , \ldots , (r_{k-1},\epsilon_{k-1})]$ of length $k-1$, and tags $(r_k, \frac{\epsilon_k}{2^n})$ such that 
\[\textstyle
[(r_1,\epsilon_1) , \ldots , (r_{k-1},\epsilon_{k-1}),(r_k,\frac{\epsilon_k}{2^n})] \in A
\]
for all $n$. By  the construction, $r_k \in Y_k$, and $Y_k$ is open, so there is an $n \in \N$ such that $(r_k,\frac{\epsilon_k}{2^n})_c \subseteq Y_k$.  But then, whenever we are employing case iv).2  after the attempt $[(r_1,\epsilon_1), \ldots , (r_{k-1},\epsilon_{k-1}),(r_k, \frac{\epsilon_k}{2^n})]$ enters $A$, we will ask if some $x \in (r_k,\frac{\epsilon_k}{2^n})_c$ is in the intersection of all $Y_m$, and if the answer is that it is not, we will not find $Y_k$ to be the guilty one. So, when all $Y_n$ are open, the tree of attempts that we are constructing will be finitely branching.
\end{itemize}
We have now proved the following theorem.
\begin{theorem}\label{thm.baire}
There is a total type 3 functional $\xi$ computable in $\IND$ and $\exists^2$ such that whenever $(Y_n)_{n \in \N}$ is a sequence of open, dense sets, then $\xi(\lambda n.Y_n) \in \bigcap_{n \in \N} Y_n$.
\end{theorem}
\noindent
By contrast, we also have the following negative result. 
\begin{theorem}\label{11.4}
There is no Baire-realiser $\xi$ computable in a functional of type 2.
\end{theorem}
We will prove this theorem by contradiction. We first have to develop some machinery and associated lemmas; then we will prove the theorem by reference to the machinery and the lemmas.

\smallskip

First of all, assume that there is an index $d$ and a type 2 functional $F$ such that for all sequences $(Y_n)_{n \in \N}$ of subsets of $\N^\N$ we have that the following:
\begin{itemize}
\item[(i)]the function $ f := \lambda a \in \N .\{d\}(F,a,(Y_n)_{n \in \N})$ is total,
\item[(ii)] if each $Y_n$ is open and dense, then $f \in \bigcap_{n \in \N}Y_n$.
\end{itemize}
We will show that this assumption leads to a contradiction, by constructing, from $d$ and $F$, a sequence $(Y_n)_{n \in \N}$ for which (i) and (ii) fail. 
The construction is based on Moschovakis' definition of \emph{computation trees} from \cite{Yannis}, but we provide most details.
As to notation we write ${\bf Y}$ for the sequence $(Y_n)_{n \in \N}$.
We also use of a modified version of S8 as in Notation \ref{notak}, which we motivate as follows.  

\smallskip

A clear disadvantage of Kleene's original S1-S9 is that the concept is restricted to functionals \emph{of pure type}. Thus, if we want to compute relative to objects of mixed types, we need to code them within the pure types. Since S8 expresses application of one object to already computed objects of lower types, the coding can make an application of this scheme hard to comprehend. 
When studying computations relative to specific objects of mixed types at level 2 or 3, a reformulation of S8 adjusted to a direct application, and not an application of the coded version, makes constructions easier to comprehend. Our modification is equivalent to how S8 would work when all objects were coded into the pure types.
\begin{nota}[Complete and incomplete computation]\label{notak}~
\begin{itemize}
\item A \emph{computation tuple} is a sequence $\langle e,\vec a; c\rangle$ indicating the terminating computation  $\{e\}(F, {\bf Y},\vec a) = c,$
where $\vec a$ is a finite sequence of numbers and we modify Kleene's S8 as follows:
\begin{itemize}
\item If $e = \langle 8,0,d\rangle$, then $\{e\}(F, {\bf Y} , \vec a) := F(f)$.
\item If $e = \langle 8,n+1, d \rangle$, then
\[
\{e\}(F, {\bf Y} , \vec a) := \left \{ \begin{array}{ccc} 0 & {\rm if} & f \in Y_n \\ 1 & {\rm if} & f \not \in Y_n\end{array} \right.,   
\]
where in both cases $f = \lambda a.\{d\}(F,{\bf Y} , a , \vec a)$.
\end{itemize}
\item In an \emph{incomplete computation tuple} $\langle e,\vec a\rangle$ we leave out the final $c$, indicating that the value of the computation is unknown \(possibly forever\).
\item The set of incomplete computation tuples  $\langle e, \vec a\rangle$ is  enumerated via a standard sequence numbering, and we let $n(\langle e,\vec a\rangle)$ be the corresponding number.
\end{itemize}
\end{nota}
Let $\varepsilon$ denote the empty sequence of integers. 
We assume throughout the construction that $e_0 = \langle 8,0,d\rangle$ is such that we for all ${\bf Y}$ have that $\{e_0\}(F,{\bf Y} , \varepsilon)\!\!\downarrow$.

\smallskip

Let us first consider the well-understood case where  ${\bf Y}$ is fixed and $\{e_0\}(F,{\bf Y} , \varepsilon)\!\!\downarrow$. Then we can find the value $c$ by building  the \emph{computation tree} for the computation by transfinite induction. We start with the top node $\langle e_0,\varepsilon\rangle$, i.e.\ an incomplete computation tuple, and then in a combined top-down and bottom-up procedure construct a tree of incomplete and complete computation tuples as explained below. In the process, we may add new incomplete computation tuples and we may turn incomplete ones to complete ones. 
We now give a semi-formal description of the aforementioned inductive process as follows. 
\begin{itemize}
\item In the case of composition as follows: 
\[
\{e\}(F,{\bf Y} , \vec a) = \{e_1\}(F,{\bf Y},\{e_2\}(F,{\bf Y} , \vec a),\vec a)
\]
we first fill in $\langle e_2 , \vec a\rangle$ as an incomplete sub-computation of $\langle e,\vec a\rangle$. When we later observe that $\langle e_2,\vec a ; b\rangle$ is the proper sub-computation, we can also fill in $\langle e_1 , b , \vec a\rangle$ as an incomplete sub-computation. When we then at an even later stage realise that $\langle e_1 , b , \vec a ; c\rangle$ is the proper sub-computation, we can make $\langle e , \vec a\rangle$ complete as $\langle e, \vec a ; c\rangle$. 
\item Primitive recursion can be seen as iterated composition, and is therefore handled in a similar way. 
\item For the rest of the schemes, it is obvious what is going on: either the incomplete computation tuple at hand is one of an initial computation, and we can fill in the correct value right away, or the set of incomplete tuples for the immediate sub-computations is uniquely given, we have to wait for the process to complete these, and then we can find the right value of the one at hand.
\end{itemize}The whole process can be seen as a simultaneous inductive definition of the construction of the tree of incomplete computation tuples and the completion of these.

\smallskip

The above describes the construction when $\bf Y$ is fixed, but in order to obtain the desired contradiction we will have to construct $\bf Y$ \emph{and} the computation tree \emph{simultaneously}, which adds complications. 
The major problem is that we do not know ${\bf Y} = (Y_n)_{n \in \N}$ when we construct the tree, but we have to make a decision what to answer whenever the procedure for constructing the computation tree requests an answer to $Y_n(\lambda a.\{d\}(F, {\bf Y} , a , \vec a))$. Our solution to this problem is that the first time $f$ in the form of $\lambda a.\{d\}(F, {\bf Y} , a , \vec a)$ is needed in our computation tree, as an input to $F$ or to some $Y_n$, we define $f \not \in Y_n$ exactly when $n = n(\langle d,\vec a \rangle)$. 

\smallskip

One useful feature of this strategy is that $Y_n$ then either is all of $\N^\N$ or just $\N^\N$ with one point missing, so $Y_n$ is open and dense. Another useful feature is that $f$ is not an acceptable value of $\xi({\bf Y})$ since $f$ is left out of one $Y_n$.  One complication is that we have to convert the tree we are constructing into a well-ordering in order to talk about e.g.\ `the first occurrence'; another complication is that we have to ensure that whenever we want to give the correct value to an S8-computation tuple, we already know which functions are used in earlier S8-computations.

\smallskip

We will now give the details of the construction. First some conventions and some intuition are needed.
\begin{defi}[Computation paths]\label{coresd}\rm~
\begin{itemize}
\item A \emph{computation path} will be a finite sequence $(t_0 , \ldots , t_{k})$ of computation tuples such that $t_{i+1}$ is a sub-computation of $t_i$ as defined below. Each $t_i$ may be complete or incomplete, but if $t_i$ is complete and $j > i$, then $t_j$ must also be complete. This reflects that we cannot give a value to a computation without knowing the values of all sub-computations.
\item In a \emph{complete computation path}, all computation tuples are complete.
\item If $t$ is a computation tuple, we will order the possible sub-computations. In this ordering, we will not discriminate between an incomplete sub-computation and its possible completion. In the process we are about to define, certain incomplete computation tuples will be turned into complete ones, and we do not want to change the position in the overall ordering. 
\item When we construct our tree by transfinite recursion, we will refer to the Kleene-Brouwer ordering based on the node-wise ordering of the sub-computations, meaning that if we extend a computation path to a longer one, we move down in the ordering.
\end{itemize}
\end{defi}
Based on Definition \ref{coresd}, we now introduce the \emph{tree of computation paths}.
We establish in Lemma \ref{11.2} below that this tree must be well-founded.
\begin{definition}[Tree of computation paths]\label{selfref}{\em
By recursion on the ordinal $\alpha$, we construct a tree $T_{\alpha}$ of computation paths as follows.  

\smallskip

First of all,  if $\alpha = 0$, we let $T_\alpha$ consist of the single computation-path $(t_0)$, where $t_0$ is the incomplete computation tuple $\langle e_0,\varepsilon)$, the computation the process aims to find the value of.

\smallskip

Secondly, if $\alpha$ is a limit ordinal, define $T_\alpha :=  \lim_{\beta < \alpha}T_\beta$.  This limit makes sense since at each step described below, either we add some incomplete sub-computation tuples at the end of a computation path that has been introduced at an earlier stage, or turn one incomplete computation tuple in the tree into a complete one.

\smallskip

Thirdly,  if $\alpha = \beta + 1$, and all computation paths in $T_\beta$ are complete, we stop. 
From now on, assume that the latter is not the case, and also assume that $T_\beta$ is well-founded, and thus well-ordered by the Kleene-Brouwer ordering we introduce in the process.
Let $(t_0 , \ldots , t_k)$ be the least element of $T_\beta$ in this Kleene-Brouwer ordering consisting of entirely incomplete computation tuples.  For our next step, we need to consider the following case distinction.
\begin{itemize}
\item If $t_k$ is a computation for S1, S2 or S3, i.e. an initial computation. Turn $t_k$ into the correct complete version and let $T_\alpha$ be the resulting tree.
\item If $t_k = \langle e, \vec a\rangle$ where $e$ is an index for $$\{e\}(F, {\bf Y} , \vec a) = \{e_1\}(F,{\bf Y} , \{e_2\}(F, {\bf Y} , \vec a))$$
By the choice of $(t_0 , \ldots , t_k)$ there will be no incomplete extension in the tree $T_\beta$. Thus there will be three subcases:
\begin{enumerate}
\item There is no extension of $(t_0 , \ldots , t_k)$ in $T_\beta$ at all.  In this case, add  $(t_0 , \ldots , t_{k+1})$ to $T_\beta$, where $t_{k+1} = \langle e_2,\vec a\rangle$.
\item There is an extension $(t_0 , \ldots , t_{k+1})$ in $T_\beta$, where $t_{k+1} = \langle e_2 , \vec a ; b\rangle$, but no extension of the form $(t_0 , \ldots , t_k,t'_{k+1})$ where $t'_{k+1} = \langle e_1,b,\vec a ; c\rangle$. Then add $(t_0 , \ldots , t_k,t''_{k+1})$ to $T_\beta$, where $t''_{k+1} = \langle e_1 , b , \vec a \rangle$. In forming the ordering of $T_\alpha$, we let this sub-computation be above the first one in our ordering of sub-computations.
\item There is an  extension $(t_0 , \ldots , t_k,t_{k+1})$ of $(t_0 , \ldots , t_k)$ in $T_\beta$ where $t_{k+1} = \langle e_2,\vec a ; b\rangle$, and an extension $(t_0, \ldots , t_k , t'_{k+1})$ in $T_\beta$ where $t'_{k+1} = \langle e_1 , b , \vec a ; c \rangle$. Then obtain $T_\alpha$ by replacing $t_{k}$ with $\langle e,\vec a ; c\rangle$ in the computation path at hand.
\end{enumerate}
\item The cases where $t_k$ is a computation for one of the schemes S5 (primitive recursion), S6 (permutation) or S9 (enumeration) are left for the reader as they just will be similar to, or simpler than, the case for S4. 
Note that S6 does not play a real role, but it gives rise to an initial computation if we allow for function arguments.
\item If $t_k$ is an S8-computation, given as $t_k = \langle e,\vec a\rangle$ where
\[
\{e\}(F,{\bf Y},\vec a) = H(\lambda a. \{d\}(F,{\bf Y} , a , \vec a)),
\] 
where either $H$ equals $ F$ or $Y_n$ for some $n$. There will be two sub-cases:
\begin{enumerate}
\item If $(t_0 , \ldots , t_k)$ has no extensions in $T_\beta$, we extend $T_\beta$ to $T_\alpha$ by adding all computation paths $(t_0 , \ldots , t_k , t_{k+1,a})$ for each $a \in \N$, where $t_{k+1,a} = \langle d,a,\vec a\rangle$. We well-order these extensions  by the value  of $a$.
\item If $(t_0, \ldots ,t_k)$ has extensions in $T_\beta$, the added computation tuple in all such extensions must be complete, by choice of $(t_0 , \ldots , t_{k})$. Moreover, since item (1) is the only way we add extensions to an S8-computation, there is a function $f$ such that we have an extension with $t_{k+1,a} = \langle d , a,\vec a ; f(a)\rangle$ for each $a \in \N$.
Now, by choice of $(t_0 , \ldots , t_k)$ again, if there is any computation path $(s_0 , \ldots , s_{j})$ in $T_\beta$ below $(t_0 , \ldots , t_k)$ in the Kleene-Brouwer ordering, where $s_j$ is an S8-computation,  $s_j$ has to be complete, and some function $g$ has been introduced. If $f$ has already been introduced as some $g$ this way, we know the value of $H(f)$ from before, and use this to make $t_k$ complete. If $f$ is introduced for the first time while we replace $T_\beta$ with $T_\alpha$, we let $n = n(\langle d,\vec a\rangle)$ as defined above, we let $f \not \in Y_n$, and $f \in Y_m$ for $m \neq n$, and use this to turn $t_k$ into a complete computation tuple for all cases of $H$. 
\end{enumerate}
\end{itemize}
This ends the construction and Definition \ref{selfref}.
}\end{definition}
We have not said what to do in the case when $T_\alpha$ is not well-founded, in which case we cannot identify the least $(t_0 , \ldots , t_k)$ where all $t_i$ are incomplete. Lemma~\ref{11.2} shows that this is never the case. The argument is based on Lemma \ref{11.1}, which has an easy proof, also under the assumption that the recursion stops when the tree $T_\alpha$ is not well-founded. 

\begin{lemma} \label{11.1}
For each ordinal $\alpha$ and $(t_0 , \ldots , t_k) \in T_\alpha$, if $t_k$ is complete, then $t_k$ is the computation tuple of a terminating computation, where ${\bf Y}$ is interpreted as the  sequence of partial sets defined at stage $\alpha$.
\end{lemma}
\begin{proof}
Trivial, by induction on $\alpha$. 
\end{proof}
\begin{lemma}\label{11.2}
For each ordinal $\alpha$, $T_\alpha$ is a well-founded tree.
\end{lemma}
\begin{proof}
We obviously need a limit ordinal $\alpha$ to introduce an infinite descending sequence.  Assume that there is one, and let $(t_0,t_1, \cdots)$ be the leftmost one. Let ${\bf Y}$ be a total extension of the sequence of partial sets $\lambda n.Y_n$ constructed at level $\alpha$. By Lemma \ref{11.1}, the sequence will consist only of incomplete computation tuples, where each extension represents a sub-computation. This will be a so-called Moschovakis witness, witnessing that $\{e_0\}(F,{\bf Y} , \varepsilon)\!\!\uparrow$, which again contradicts the assumption. 

\smallskip

We need Lemma \ref{11.1} in order to verify that the sequence is indeed a Moschovakis witness when passing an instance of composition (or primitive recursion). Note that in the presence of S9, we do not need the scheme S5, primitive recursion, so for the understanding, one may ignore this case.
\end{proof}
We can now complete the proof of Theorem \ref{11.4}.
\begin{proof}
Assume that the theorem is false. Then there is an $F$ and an index $d$ such that for all ${\bf Y}$ we have $\xi({\bf Y}) = \lambda a. \{d\}(F,{\bf Y},a)$.
Let $e_0$ be an index such that 
\[
\{e_0\}(F,{\bf Y} , \varepsilon) = F(\lambda a.\{d\}(F,{\bf Y},a)).
\]
When we apply our construction above to this $e_0$, we construct a ${\bf Y}$ where each $Y_n$ is open and dense, but where every function $f$ appearing in the computation tree of $\{e_0\}(F, {\bf Y} , \varepsilon)$ is left out of exactly one $Y_n$. This will in particular be the case for $\lambda a.\{d\}(F,{\bf Y},a)$, so this is not an acceptable value for $\xi({\bf Y})$ after all.
\end{proof}
We finish this section with a remark on future research. 
\begin{remark}{\em First of all, the proof of Theorem \ref{11.4} is very different from known proofs of theorems expressing that certain type 3 functionals are not computable in any functional of type 2; we refer to \cite{dagsam,dagsamII, dagsamIII, dagsamV} for the latter kind of proofs. 

\smallskip

Secondly, Baire realisers are not unique, but as shown in Theorem \ref{thm.baire}, there is a specimen computable in $\IND$. This begs the question of the necessary complexity of Baire realisers, and how they compare to realisers for $\HBU$ and Pincherle's Theorem, as in \cites{dagsam,dagsamII, dagsamIII, dagsamV}. We have no answer to this, and offer it as a research problem. 
}\end{remark}

\subsubsection{Reverse Mathematics}
We prove the results sketched in items (c) and (d) in Section \ref{BCT1}.  
We first introduce our version of the Baire category theorem, called $\BCT$ hereafter, based on Definition \ref{openset}.
\begin{thm}[$\BCT$]
If $ (Y_n)_{n \in \N}$ is a sequence of dense open sets of reals, then 
$\emptyset \ne \bigcap_{n \in\N } Y_n$.
\end{thm}
We have the following theorem, establishing that $\BCT$ does not exceed $\ACA_{0}$ in terms of first-order strength. 
\begin{thm}\label{thmke}
The system $\RCAo+\BOOT$ proves $\BCT$.
\end{thm}
\begin{proof}
In case $\neg(\exists^{2})$, all functions on $\R$ are continuous by \cite{kohlenbach2}*{\S3}.  
Such functions have RM-codes by \cite{kohlenbach4}*{\S4}, as $\ACA_{0}$ follows from $\BOOT$.
In turn, all open sets now also have RM-codes by \cite{simpson2}*{II.7.1} and $\BCT$ is therefore
immediate from the second-order proof of the Baire category theorem (see \cite{simpson2}*{II.4.10}).   

\smallskip

In case $(\exists^{2})$, let $Y_{n}$ be as in the theorem and consider the following formula
\be\label{bobsleeps}
(\forall y\in \R)( y\in B(q_{m}, r_{m})\di y\in Y_{n}), 
\ee
which expresses $B(q_{m}, r_{m})\subseteq Y_{n}$.
Applying $\BOOT$ to \eqref{bobsleeps} (using $\exists^{2}$ to remove numerical quantifiers), there is $Z\subseteq \N\times \N$ such that $(m,n)\in Z$ if and only if \eqref{bobsleeps} holds. 
Define the sequence $(X_{n})_{n\in \N}$ as $\lambda n.Z$ and note that $x\in Y_{n}\asa (\exists m\in \N)(m\in X_{n}\wedge x\in B(q_{m}, r_{m}))$.  Hence, $Y_{n}$ also has an RM-code in this case and 
$\BCT$ is therefore immediate from the aforementioned second-order proof of the Baire category theorem (see \cite{simpson2}*{II.4.10}).   
\end{proof}
Let $\BOOT_{\w}$ be $\BOOT$ with the (innermost) quantifier over Baire space restricted to $2^{\N}$.  
As shown in \cite{samph}*{\S5}, $\BOOT_{\w}$ has first-order strength at most $\WKL_{0}$, as does the base theory in the following corollary; the proof of the latter is immediate. 
\begin{cor}
The system $\RCAo+\BOOT_{\w}+\cont$ proves $\BCT$.
\end{cor}
Another corollary of Theorem \ref{thmke} is that $(\exists^{3})$ implies $\BCT$, while $\exists^{3}$ computes the real claimed to exist by $\BCT$ from the other data.  
The proof of Theorem~\ref{thmke} however also gives rise to the following classical result, where `$\leq_{T}$' means Turing computability.
Let $J(Y)=\{n\in \N: (\exists f\in \N^{\N})(Y(f, n)=0)\}$ be the set $X$ as provided by $\BOOT$; there is a term $t$ of G\"odel's $T$ such that we have:
\[
\text{for dense open sets $(Y_{n})_{n\in \N}$ in $\R$, there is $x\in\cap_{n}Y_{n}$ with $x\leq_{T} J(t(Y_{n}, \exists^{2}))$.  }
\]
The previous statement can be viewed as the `second-order' version of the fact that $\exists^{3}$ computes a Baire realiser, in the sense of Kleene S1-S9. 
We also have the following negative result involving an interesting corollary.  
\begin{thm}
The system $\Z_{2}^{\omega} + \QFAC^{0,1}$ cannot prove $\BCT$.
\end{thm}
\begin{proof} 
We shall define a model $\mathcal{M}$ of $\Z_{2}^{\omega} + \QFAC^{0,1}+\neg\BCT$.  Incidentally, $\mathcal{M}$ has been used to establish \cite{dagsamV}*{Theorem 4.3}, as well as \cite{dagsamVIII}*{Theorem 2.2}.
The model $\mathcal{M}$ is an element of G\"odel's universe $\textsf{L}$ of constructible sets, and the properties will be verified under the assumption that \textsf{V = L}. 
Since being a model for the theory in question is absolute, the proof of existence can be formalised in \ZF.

\smallskip

Thus, assume $\textsf{V = L}$,  let $\SS^{2}_{\omega}$ be the join of all the functionals $\SS^{2}_{k}$, and let $\mathcal M$ be the least type-structure closed under computability relative to $\SS^{2}_{\omega}$. 
Let $\mathcal{M}_{n}$ be the set of elements in $\mathcal M$ of type $n$.
We assume $\textsf{V = L}$ because then there is a $\Delta^1_2$-well-ordering of the continuum; this can be used to show that all $\Pi^1_k$-formulas are absolute for $A = {\mathcal M}_1$. 
Hence, $\mathcal M$ is a model for $\Z_{2}^{\omega}$ and also satisfies $\QFAC^{0,1}$ as a consequence of Gandy selection for $\SS^2_\omega$.

\smallskip

We now show that ${\mathcal M}_2$ contains a functional $F:A \rightarrow \N$ that is \emph{injective}. This construct is obtained by Gandy selection as follows: for each $f \in A$ we let $F(f)$ be an index for computing $f$ from $\SS^2_\omega$. We will use this to show that $\BCT$ fails for $\N^\N$ in $\mathcal{M}$, but we can use the same idea for Cantor-space or for $\R$ (in $\mathcal{M}$). 

\smallskip

We now work inside $\mathcal M$.  
Define $O_n = \{f : F(f) > n\}$ and note that since the complement of each $O_n$ is finite, each $O_n$ is open and dense. Moreover, $\{(n,f) : f \in O_n\}$ is definable from $F$ by a term in G\"odel's $T$, so this countable sequence of dense, open sets will be in $\mathcal M$. The intersection is empty, so $\BCT$ fails in $\mathcal M$.
\end{proof}
The model $\mathcal{M}$ has been used in \cite{dagsamVIII}*{Theorem 2.2} to show that the statement
\be\tag{\textup{\textsf{NIN}}}
(\forall Y:[0,1]\di \N)(\exists x, y\in [0,1])(Y(x)=_{0}Y(y) \wedge x\ne_{\R}y)
\ee
cannot be proved in $\Z_{2}^{\omega}+\QFAC^{0,1}$.  Note that $\NIN$ expresses that there is no injection from $[0,1]$ to $\N$.
By contrast $\WHBU\di \textsf{NIN}$ and $\BCT\di \textsf{NIN}$ over $\RCAo$, as also proved in \cite{dagsamVIII}.

\section{A finer study of representations of open sets}\label{waycool}
We establish the results sketched in Section \ref{rmintro} regarding the $\Delta$-functional; the latter connects two natural representations of open sets and is introduced in Section~\ref{gintro}.
The main computational properties of the $\Delta$-functional are established in Sections \ref{RDRR2} and \ref{RDRR}, while some related RM results are sketched.
In a nutshell, these results show that the representation \eqref{R2} of open sets does not have much of an influence on the logical and computational properties of theorems pertaining to open sets, at least in contrast to Definition \ref{openset}.
\subsection{Introduction}\label{gintro}
In the previous, we have considered two different representations of open sets, namely the standard (RM) one as in \eqref{morg} and the approach via characteristic functions as in Definition \ref{openset}.
There are of course other possible representations, namely as part of \eqref{R1}-\eqref{R4} below; in this section, we show that \eqref{R3} and \eqref{R4} are computationally equivalent, and study the computational properties 
of the `conversion' functional $\Delta$ that converts a representation as in \eqref{R2} to a representation as in \eqref{R3}.  The $\Delta$-functional has interesting properties, as follows.
\begin{enumerate}
\item[(P1)] $\Delta$ is not computable in any type $2$ functional, but computable in any Pincherle realiser, a class weaker than $\Theta$-functionals (Theorem \ref{XNX}).
\item[(P2)] $\Delta$ is unique, genuinely type $3$, and adds no computational strength to $\exists^2$ in terms of computing functions from functions (Corollary \ref{sweet}).
\end{enumerate}
Prior to the study of \eqref{R2} and \eqref{R3}, we believed that the only way to find a functional with properties (P1) and (P2) would be through some ad hoc construction and that there would be no natural examples. 
Regarding RM, we also show that $\HBU$ suffices to show that open sets as in \eqref{R2} have a representation as in \eqref{R4}.  

\smallskip

For the sake of simplicity, we restrict the attention to $[0,1]$, though the $\sigma$-compactness of $\R$ makes it easy to extend all results to $\R$.
Thus, we consider the following four ways of representing an open set $O$ in $[0,1]$. For the sake of notational simplicity, we let $(a,b)$ denote $(a,b) \cap [0,1]$.
Finally, we could obtain the same results for $2^{\N}$ instead of $[0,1]$, i.e.\ our results do not really depend on the coding of the real numbers.  
\begin{enumerate}
\renewcommand{\theenumi}{R.\arabic{enumi}}
\item The set $O$ is $ \{x \in [0,1]\mid Y(x) >_{\R} 0\}$ for some $Y:[0,1]\di [0,1]$; we just have the extra information that $O$ is open, i.e.\ as in Definition \ref{openset}.\label{R1}
\item The set $O$ is represented by a function $Y : [0,1] \rightarrow [0,1]$ such that
\begin{itemize}
\item[(i)] we have $O = \{x \in [0,1]\mid Y(x) >_{\R} 0\}$, 
\item[(ii)] if $Y(x) > 0$ , then $(x-Y(x),x+Y(x))\cap [0,1] \subseteq O$. 
\end{itemize}\label{R2}
\item The set $O$ is represented by the \emph{continuous} function $Y$ where
\begin{itemize} 
\item[(i)] $Y(x)$ is the distance from $x$ to $[0,1]\setminus O$ if the latter is nonempty,
\item[(ii)] $Y$ is constant 1 if $O = [0,1]$.
\end{itemize}\label{R3}
\item The set $O$ is given as $\cup_{n\in \N}(a_{n}, b_{n})$ for $(a_m)_{m\in \N}, (b_{m})_{m\in\N}$ sequences in $\Q$. %If $a_i < 0$ or $b_i > 1$ we, by convention, cut off at the end points of $[0,1]$.
\label{R4}
\end{enumerate}
Assuming $\exists^2$, it is clear that the information given by a representation increases when going down the list.   
For completeness, we prove that \eqref{R3} and \eqref{R4} are the same from the computational point of view. 
\begin{theorem}\label{friuk}
Items \eqref{R3} and \eqref{R4} are computationally equivalent modulo $\exists^2$.
\end{theorem}
\begin{proof}
Let $Y$ be continuous as in \eqref{R3}. Then $Y$ has an RM-code computable in $\exists^{2}$ by \cite{kohlenbach4}*{\S4}. From this representation we can decide if $Y$ is constant 1 or if $Y(x) = 0$ for at least one $x$. Let $\alpha$ be this representation for $Y$. Then $x \in O$ if and only if there is some $((a,b),(c,d)) \in \alpha$ such that $c > 0$ and $x \in (a,b)$, and the set of  intervals $(a,b)$ where $((a,b),(c,d)) \in \alpha$ for some $(c,d)$ with $c > 0$ will be a representation of $O$ in the sense of \eqref{R4}.

\smallskip

Now assume that $A$ is a set of open rational intervals defining $O$ as in \eqref{R4}. 
Then $O = [0,1]$ if and only if $A$ contains a finite sub-covering of $[0,1]$. 
If it does, we let $Y$ be the constant 1. If not, let $x$ be given. If $x \not \in (a,b)$ for all $(a,b) \in A$, we let $Y(x) = 0$. 
If $x \in (a,b)$ for some $(a,b) \in A$, we let $Y(x)$ be the supremum of the set of rationals $r$ such that $A$ contains a finite sub-covering of $[x-r,x+r]$.
\end{proof}
In second-order RM, $\ACA_{0}$ is equivalent to the fact that closed sets are located (\cite{withgusto}*{Theorem~1.2}).  
The previous theorem similarly expresses that a set is open if and only if the complement is located.
Next, we study the computational relation between the representations defined by \eqref{R2} and \eqref{R3}.

\subsection{Converting between representations}\label{RDRR2}
In this section, we study the complexity of operators that produce a representation as in \eqref{R3}, or equivalently by Theorem \ref{friuk}: as in \eqref{R4}, from a representation as in \eqref{R2}. 
We (mostly) choose to study \eqref{R3} as this representation is \emph{unique} for each open set, resulting in a \emph{unique} functional, i.e.\ the output of $\Delta$ is unique for equal inputs satisfying \eqref{R3}.
%, as follows.
\begin{definition}
Let $\Delta^{3}$ be the functional such that $\Delta(Y)$ represents an open set $O$ as in \eqref{R3} whenever $Y$ represents $O$ as in \eqref{R2}.
\end{definition}
The following proof is straightforward in light of similar proofs in \cites{dagsam, dagsamII, dagsamIII, dagsamV}, and we therefore only provide a sketch. 
\begin{lemma} 
The functional $\Delta$ is not computable in any type 2 functional.
\end{lemma}
\begin{proof}
Given $F^2$, we construct  $Y^{2}$ with the following properties.
\begin{itemize}
\item The value $Y(f)$ is defined if $f$ represents a fast-converging sequence of rational numbers in $[0,1]$,
\item If the sequence represented by $f$ is equivalent to a sequence represented by some $g$ computable in $F$, we use Gandy selection for $F$ to find an index $e$ for one such $g$ as computable in $F$.  
Note that the Gandy-search is such that the resulting $g$, and index for it,  respects equivalence between representations of reals.  We then define $Y(f) = 2^{-(e+2)}$ for the aforementioned index $e$.
\end{itemize}
The crux of the previous construction is as follows: the functional $Y$ represents $O = [0,1]$ but no Kleene-algorithm relative to $F$ and $Y$ is able to recognise this.  Indeed, since $Y$ is partially computable in $F$ when restricted to functions computable in $F$, it only covers a subset of measure below $ \frac{1}{2}$ in this situation.
\end{proof}
Next, we show that $\Delta$ is computable from $\exists^{2}$ and a \emph{Pincherle realiser}, i.e.\ a realiser for Pincherle's theorem, a concept first introduced in \cite{dagsamV}.
The latter theorem expresses that a locally bounded functional on $C$ is also bounded (see \cite{tepelpinch}*{p.~67}); a \emph{Pincherle realiser} (PR for short) computes this upper bound in terms
of some of the other data.  We consider the following equivalent form of Pincherle realisers, going back to an equivalent formulation by Pincherle himself (see \cite{tepelpinch, dagsamV}).  
Note that we assume that $Y$ is extensional on the reals. 
\begin{definition}
A PR is any functional $M_u^{3}$ such that 
for any $Y:[0,1] \rightarrow \R^{+}$, the number $M_u(Y) >_{\R} 0$ is a lower bound for all $Z:[0,1] \rightarrow \R$ locally bounded away from zero by $Y$, i.e.\ satisfying $\(\forall x,y \in [0,1])(|x-y| < Y(x) \rightarrow Z(y) > Y(x))$.
\end{definition}
Note that the functional $Y$ is a realiser for `$Z$ is locally bounded from zero'.  As discussed in \cite{dagsamV}, Pincherle assumes the existence of such realisers (for local boundedness) in \cite{tepelpinch}*{p.\ 66}.
The following theorem is interesting, as PRs are %strictly 
 weaker than $\Theta$-functionals (realisers for $\HBU$), as shown in \cite{dagsamV}.
\begin{theorem}\label{XNX}
The functional $\Delta$ is computable in any PR and $\exists^2$.
\end{theorem}
\begin{proof}
Let $Y$ represent the open set $O$ as in \eqref{R2}. We  first show that any PR $M_u$ and $\exists^2$ allows us to decide if $O = [0,1]$ or not.

\smallskip

Define $Y_n(x)$ as $ Y(x)$ if $Y(x) _{\R}> 0$ and $ 2^{-n}$ otherwise. If $O = [0,1]$ then $Y_n = Y$ for all $n$, and $M_u(Y_n)$ is positive and independent of $n$. If $x_0 \not \in O$, there is no $x$ such that $Y(x) > 0$ and $x_0$ is in the neighbourhood around $x$ defined by $Y(x)$.  Hence, $2^{-n}$ will be the lower bound on $Z(x_0)$ induced by $Y_n$, yielding $0 < M_u(Y_n) \leq 2^{-n}$. 
Hence, we can decide if $O = [0,1]$ or not using the sequence  $\lambda n.M_u(Y_n)$.

\smallskip

Next consider $x \in [0,1]$ and assume that there is some (unknown) $z \not \in O$. For each rational $r > 0$, we define the following set:
\[
O_{x,r} = O \cup \{y \in [0,1] : |x-y] > r\},
\]
and we let $Y_{x,r}$ be the representation of $O_{x,r}$ provided by \eqref{R2}. Now let $Z(x)$ be the distance from $x$ to the complement of $O$. 
Then $Z(x) \leq r$ if and only if $O_{x,r} \neq [0,1]$, and we can use $M_u$ as above deciding this for each $x\in [0,1]$ and $r$. 
Since we can decide if $Z(x) \leq r$ uniformly in $r$, we can use $\exists^2$ to compute $Z(x)$.
\end{proof}
We have the following corollary, establishing  the above claims concerning the $\Delta$-functional introduced in the previous section.
\begin{corollary}\label{sweet}
The functional $\Delta$ is computationally `weak' as follows.
\begin{itemize}
\item[(a)]
For each $f \in \N^\N$, all functions computable in $\Delta , \exists^2 , f$ are also hyperarithmetical in $f$ alone.
\item[(b)] There is no PR $M_u$ computable in $\Delta$ and $\exists^2$.
\end{itemize}
\end{corollary}
\begin{proof}
Since any $\Theta$-functional computes some PR, $\Delta$ is uniformly computable in any $\Theta$-functional, and the only functions that are uniformly computable in any $\Theta$ and $\exists^2$ are the hyperarithmetical ones (see \cite{dagsamIII, dagsamV}).   This readily relativises to any function $f$, i.e.\ the first item follows.

\smallskip

Each PR $M_u$ computes (non-uniformly) a non-hyperarithmetical function, while $\Delta$ only computes hyperarithmetical ones, i.e.\ no $M_u$ can be computable in $\Delta$. 
\end{proof}
For completeness, we obtain an RM result pertaining to \eqref{R2} and \eqref{R4}.
Now, it is shown in \cite{dagsamV} that a \emph{uniform} version of Pincherle's theorem (which forms the basis for PRs) and $\HBU$ are equivalent over a weak base theory.  
In light of Theorem~\ref{XNX}, one therefore expects $\HBU$ to suffice to prove the following coding principle, as shown in Theorem \ref{kore}.  Corollary~\ref{sweet} suggests no reversal can be obtained.  
\bdefi[$\open^{-}$]
An open set as in \eqref{R2} has a representation as in \eqref{R4}.
\edefi
The following theorem should be contrasted with \cite{samph}*{Theorem 4.4} where it is shown that a generalisation of $\open$ and $\open^{-}$ implies $\BOOT$ (which implies $\HBU$).
Note that the following proof makes explicit use of the extra information that \eqref{R2} includes over \eqref{R1}, i.e.\ the latter cannot be treated in the same way.  
\begin{thm}\label{kore}
The system $\RCAo+\HBU+ \QFAC^{0,1}$ proves $\open^{-}$.
\end{thm}
\begin{proof}
We may assume $(\exists^{2})$ as the theorem is trivial in case $\neg(\exists^{2})$.
Indeed, all functions on $\R$ are continuous in this case by \cite{kohlenbach2}*{\S3}, while $\WKL$
suffices to provide RM-codes on $[0,1]$ by \cite{kohlenbach4}*{\S4}.    Then \cite{simpson2}*{II.7.1} yields the required RM-open set.  

\smallskip

Let $O$ and $Y$ be as in \eqref{R2} and consider the following for any $ n\in\N, q\in \Q\cap[0,1]$:
\be\label{JB}\textstyle
 \overline{B}(q, \frac{1}{2^{n}})\subset O \asa  (\exists y_{0}, \dots, y_{k}\in O)\big[ \overline{B}(q, \frac{1}{2^{n}})\subset \cup_{i\leq k} B(y_{i} , Y(y_{i}))\big],
\ee
where $\overline{B}(x, r)$ is the closed ball with center $x$ and radius $r>_{\R}0$.  The reverse implication in \eqref{JB} follows by definition.  
The forward implication in \eqref{JB} follows from applying $\HBU$ to the uncountable covering $\cup_{x\in \overline{B}(q, \frac{1}{2^{n}})}B(x, Y(x))$ of   $ \overline{B}(q, \frac{1}{2^{n}})$.

\smallskip

Now note that the formula in big square brackets in \eqref{JB} is decidable thanks to $\exists^{2}$, while `$ \overline{B}(q, \frac{1}{2^{n}})\subset O$' has the form $(\forall y\in \R)B(y, q, n)$ where $B$ is also decidable modulo $\exists^{2}$.
In other words, \eqref{JB} is equivalent to a formula of the form $(\forall m^{0})((\forall y\in \R)B(y, m)\asa (\exists x\in \R)A(x, m) )$.  Clearly, $\QFAC^{0,1}$ applies to $(\forall m^{0})((\forall y\in \R)B(y, m)\di (\exists x\in \R)A(x, m) )$, yielding $\Phi^{0\di 1}$ such that 
$(\forall m^{0})((\forall y\in \R)B(y, m)\asa (\exists x\in \R)A(x, m) )\asa A(\Phi(m), m)$.  In this way, there is a set $X\subset \Q\times \N$ such that for any $q\in\Q\cap [0,1], n\in \N$, we have
\be\label{JB2}\textstyle
(q, n)\in X  \asa  (\exists y_{0}, \dots, y_{k}\in O)\big[ \overline{B}(q, \frac{1}{2^{n}})\subset \cup_{i\leq k} B(y_{i} , Y(y_{i}))\big].
\ee
Again modulo $\exists^{2}$, apply $\QFAC^{0,1}$ to the forward implication in \eqref{JB} to find $\Psi^{0\di 1^{*}}$ such that if $(q, n)\in X$, $\Psi(q, n)=\langle y_{0}, \dots,  y_{k}\rangle$ is as in the right-hand side of \eqref{JB2}.
Now consider the following formula based on the previous:
\be\label{JB3}
x\in O\asa (\exists n\in \N)(\exists q\in \Q\cap [0,1])((q, n)\in X \wedge x\in  O_{q, n} ), 
\ee
where $O_{q,n}$ is $\cup_{i<|\Psi(q, n)|}  B(\Psi(q, n)(i), Y(\Psi(q, n)(i))  )$.  The reverse implication in \eqref{JB3} follows by \eqref{JB} and \eqref{JB2}.  
For the forward implication in \eqref{JB3}, fix $x_{0}\in O$ and let $n_{0}$ be large enough to guarantee $\overline{B}(q_{0}, \frac{1}{2^{n_{0}}}  )\subset O$ for $q_{0}:=[x_{0}](n_{0})$.
By \eqref{JB} and \eqref{JB2}, we have $(q_{0}, n_{0})\in X$ and thus $\overline{B}(q_{0}, \frac{1}{2^{n_{0}}})\subset O_{q_{0}, n_{0}}$.  By assumption, $x_{0}$ is in $\overline{B}(q_{0}, \frac{1}{2^{n_{0}}})$ and \eqref{JB3} thus provides the countable union required by \eqref{R4}.
\end{proof}
We note that \eqref{JB2} is a special case of a new comprehension axiom called `$\Delta$-comprehension' or `$\Delta\text{-}\textsf{CA}$' in \cite{samph, samFLO2, samrecount}; this axiom 
yields the well-known recursive comprehension axiom from $\RCA_{0}$ under $\ECF$.  There is no connection to the $\Delta$-functional beyond a shared symbol.  We also note that $\HBU$ 
can be replaced by the Lindel\"of lemma for $\R$, as studied in \cite{dagsamV}. 

\smallskip

As a conceptual corollary, the previous theorem shows that the RM of $\HBU$ generalises to theorems pertaining to open and closed sets as in \eqref{R2}.  

\subsection{Representations, Heine-Borel, and Baire}\label{RDRR}
We establish two theorems on the computational complexity of the countable Heine-Borel theorem and the Baire category theorem when formulated using the representation \eqref{R2}.  % from Section \ref{gintro}.

\smallskip

First of all, $\Delta+\exists^{2}$ suffices to compute finite sub-coverings for coverings formulated using \eqref{R2}; no type two functional can replace $\Delta$ here. 
\begin{theorem}\label{cikometric}
Let $Y_n$ be a representation of the open set $O_n \subseteq [0,1]$ as in \eqref{R2} and assume $[0,1]\subseteq\cup_{n\in \N}O_{n}$. Then  $\Delta+ \exists^2$ computes $k\in \N$ such that
$[0,1]\subseteq \cup_{i\leq k}O_{i}$. The functional $\Delta$ cannot be replaced by any functional of type 2 here. 
\end{theorem}
\begin{proof}
Using $\Delta$ on each $Y_n$, and the equivalence between \eqref{R3} and \eqref{R4}, we obtain a standard RM-covering of the unit interval, and then we only need a realiser for $\WKL$, computable in $\exists^2$,  to prove the first claim.

\smallskip

For the second claim, we  use one of our standard techniques: given $F^2$, assume that the realiser $\beta$ for $\HBC$ (formulated with \eqref{R2}) is computable in $F+\exists^2$.  
We let each $Y_n$ be the same: it represent $[0,1]$ as an open set but in such a way that $Y_n$ restricted to the reals computable in $F+\exists^2$ is partially computable in $F+\exists^2$ and only defines an open set of measure $\leq \frac{1}{2}$. 
Let $k = \beta(\lambda n.Y_n)$, and let $Y_i'$ be derived from $Y_i$ for $i \leq k+1$ by removing one, and the same,  point that is outside this set. Let $Y_i' = Y_i$ for $i > k+1$. Then $\beta(\lambda n.Y_n) = \beta(\lambda n.Y_n')$, contradicting what $\beta$ is supposed to achieve.
\end{proof}
The first part of this proof also implies that the optimal realiser for $\HBC$ (formulated using \eqref{R2}), the one selecting the least $k$ as in Theorem \ref{cikometric}, is equivalent to $\Delta$, given $\exists ^2$. We leave the proof of this to the reader.

\smallskip

Secondly, the Baire category theorem becomes a lot `tamer' from the computational point of view upon the introduction of \eqref{R2}.
\begin{theorem}
Assuming a representation as in \eqref{R2} is given for all the $Y_{n}$, we can compute a Baire realiser relative to $\exists^2$. 
If we know that each $Y_n$ represents a dense open set, we can even avoid $\exists^2$.
\end{theorem}
\begin{proof} Let $Y_n$ be given for each $n$ and first assume that each $Y_n$ is an \eqref{R2} representation of the open, dense set $O_n$. Then the text-book proof of the Baire category theorem  can be transformed to an algorithm by using $Y_n$ restricted to the rational numbers to find the shrinking sequence of open-and then-closed intervals, and iterated search over $\Q$ for the sequence that converges to a point in the intersection. 

\smallskip

We need to search for $q \in \Q$ within an open interval such that $Y_n(q) > 0$, but since the latter relation is $\Sigma^0_1$, this is effective. If one $Y_n$ does not represent an open, dense set, the only thing that might  prevent this algorithm from terminating is that the search goes on for ever, and whether this will be the case can be decided using $\exists^2$, so we can output a value even then. 
In this case,  it does not matter what the value is, we are only required to have one.
\end{proof}

\begin{ack}\rm
Our research was supported by the \emph{John Templeton Foundation} via the grant \emph{a new dawn of intuitionism} with ID 60842 and by the \emph{Deutsche Forschung Gemeinschaft} (DFG) via grant SA3418/1-1.
The results in Section~\ref{limp3} were completed during the stimulating BIRS workshop (19w5111) on Reverse Mathematics at CMO, Oaxaca, Mexico in Sept.\ 2019.  
We express our gratitude towards the aforementioned institutions. 
We thank Anil Nerode and the referees for their helpful suggestions.  
Opinions expressed in this paper do not necessarily reflect those of the John Templeton Foundation.    
\end{ack}

\bye